\documentclass[11pt]{article}

\usepackage{amsmath,amssymb,amsfonts,amsmath}
\usepackage{graphicx}
\usepackage{subfigure}
\usepackage{caption}
\usepackage[title]{appendix}
\usepackage{float}
\usepackage[dvipdf]{epsfig}
\usepackage{eqparbox}

\usepackage{multicol}               % used for the two-column index
\usepackage{verbatim}
\usepackage{color}
\numberwithin{equation}{section}

\newtheorem{lem}{Lemma}[section]
\newtheorem{thm}{Theorem}[section]

\newenvironment{proof}{{\bf Proof: \/}}
{\vspace{-.3in}\begin{flushright}{\vrule height5pt width3pt
depth0pt} \end{flushright}}
{\vspace{-.3in}\begin{flushright}{\vrule height5pt width3pt
depth0pt} \end{flushright}}

\newcommand{\eqn}[1]{
\begin{eqnarray}
#1
\end{eqnarray}
}

\newcommand{\tb}[1]{
	\textbf{#1}
}

\newcommand{\norm}[1]{
	\left\lVert #1 \right\rVert
}

\newcommand{\ip}[1]{
	\left\langle #1\right\rangle
}

\newcommand{\pr}{
	\partial_{\rho}
}

\newcommand{\pt}{
	\partial_{\theta}
}

\newcommand{\pti}{
	\partial_{t}
}

\newcommand{\pz}{
	\partial_{z}
}

\newcommand{\pzp}{
	\partial_{z'}
}

\newcommand{\di}{
	\mbox{div}
}

\begin{document}

\title{Computing the Dirichlet-Neumann Operator on a Cylinder \thanks{Accepted for publication in SIAM Journal on Numerical Analysis.} }
\author{Saad Qadeer \thanks{Department of Mathematics, University of North Carolina Chapel Hill
		(saadq@email.unc.edu)} \and Jon Wilkening \thanks{Department of Mathematics, University of
		California Berkeley ({wilken@math.berkeley.edu}). This work
		was supported in part by the National Science Foundation under
		award number DMS-1716560, and by the U.S. Department of Energy,
		Office of Science, Applied Scientific Computing Research, under
		award number DE-AC02-05CH11231.}}
\date{March 18, 2019}

\maketitle

\begin{abstract}
	The computation of the Dirichlet-Neumann operator for the Laplace equation is the primary challenge for the numerical simulation of the ideal fluid equations. The techniques used commonly for 2D fluids, such as conformal mapping and boundary integral methods, fail to generalize suitably to 3D. In this study, we address this problem by developing a Transformed Field Expansion method for computing the Dirichlet-Neumann
	operator in a cylindrical geometry with a variable upper boundary. This technique reduces the problem to a sequence of Poisson equations on a flat geometry. We design a fast and accurate solver for these sub-problems, a key ingredient being the use of Zernike polynomials for the circular cross-section instead of the traditional Bessel functions. This lends spectral accuracy to the method as well as allowing significant computational speed-up. We rigorously analyze the algorithm and prove its applicability to a wide class of problems before demonstrating its effectiveness numerically.
\end{abstract}

\section{Introduction}

Water-wave equations are notoriously hard to solve numerically because of the nonlinear nature of the problem and the evolving domain, which is itself an unknown quantity. Modern formulations of this problem have focused on the evolution of the boundary variables with the information from the interior of the domain obtained with the help of a Dirichlet-Neumann operator. While the problem has been proven to be well-posed \cite{lannes}, the non-locality of the DNO has been noted to pose a severe challenge in both numerical and theoretical studies \cite{harrop,kaknic09,wilk5w}.   

Traditionally, numerical computations of the DNO have been restricted to the 2D case. Several elegant and robust numerical techniques have been devised including, among others, conformal mapping, finite element and boundary integral methods. However, these techniques do not carry over successfully to higher dimensions since they either rely inherently on the geometry of a 2D space or scale poorly with dimension. We therefore need to consider approaches that may not be widely used in 2D but can be extended to 3D. In \cite{wilk5w}, the authors exhaustively analyze a number of these, including the operator expansion of Craig \& Sulem \cite{crsulem}, the integral equation formulation of Ablowitz, Fokas \& Musslimani \cite{afm} and the transformed field expansion method (TFE) of Nicholls \& Reitich \cite{nicreit01,nicreit04,nicreit06}. The first two are shown to suffer from catastrophic numerical instabilities which severely limits their utility. In particular, they involve significant cancellations of terms or, equivalently, a rapid decay of singular values of the truncations of the associated linear operators. As a result, these methods require multiple precision arithmetic to yield accurate solutions.   

The TFE method, on the other hand, possesses a straightforward generalization to 3D and yields a numerically stable high-order algorithm. In addition, it is also able to handle artificial dissipation \cite{kaknic09}. A careful analysis of this technique, as applied to problems from fluid mechanics and acoustics, is presented in \cite{nicshen09}, including proofs of the analytic dependence of the DNO on the wave profile and the convergence of the method. One shortcoming of this approach is that it is unable to capture a surface bending back on itself as it requires the interface to be the graph of a function. For most applications, however, this is not an issue.  

In this study, we generalize the TFE method to build a solver for the DNO problem for Laplace's equation on a cylinder. We specifically work in this geometry as it poses the most significant computational challenges out of all the regular geometries over a plane. Rectangular geometries with periodic boundary conditions essentially avoid the issue of addressing the interaction of the fluid with a wall and, at any rate, can be treated similarly by an extension of this technique. Previous applications of this technique to 3D have concentrated mainly on the computation of electromagnetic and acoustic waves in spherical and periodic domains (see for instance \cite{fang07,li2016near,ma2015spectral}). The traditional Field Expansion method introduced in \cite{brreit92} proceeds by assuming the domain is a perturbation of a flat geometry and expanding the DNO in terms of the perturbation parameter. Naive implementations of this approach however lead to large cancellations, making it unsuitable for numerical procedures. These cancellations can be avoided by first flattening out the domain; this also simplifies the geometry of the problem and the PDE to be solved is replaced by a sequence of related problems \cite{nichollshops}. Successful implementations of this technique therefore require the rapid computation of the solutions to the associated problems.

The following sections are organized as follows. We first state the water-wave problem and the role of the DNO in the surface formulation. Next, we generalize the TFE framework to a cylindrical geometry and obtain a system of associated problems. A numerical method, based on Zernike polynomials, is presented next for solving these equations. We also provide some implementation details and determine the computational complexity of the algorithm. In the next section, a rigorous analysis of the numerical method is presented, including analyticity results for the TFE method and error estimates for Zernike representations. These are confirmed numerically in the following section before demonstrating the effectiveness of the algorithm on a number of DNO test problems. We conclude by providing further insights and extensions while also describing the limitations of this approach.

\section{The Problem}
Consider a cylinder of unit radius with a flat bottom containing an incompressible, irrotational and inviscid fluid. Denote the Cartesian coordinates on this geometry by $(x',y',z')$. Suppose the fluid at rest has a depth of $h$ (written $z' = -h$) while the interface at the top is given by $z' = \eta(x',y')$. We assume that $h > \norm{\min \{\eta,0\}}_{\infty}$. % $h > \norm{\eta}_{\infty}$). 

The irrotationality of the fluid allows us to express its velocity at any point as the gradient of a potential function $\phi$. The evolution of the fluid is then described by Euler's equations \cite{lamb}: 
\eqn{
	\Delta' \phi &=& 0 \ \ \ \ \ \ \ \ \ \    -h < z' < \eta \label{incomp}\\
	\pti \phi +  \frac{1}{2}|\nabla' \phi|^2 + (g - F(t))\eta &=& 0 \ \ \ \ \ \ \ \ \ \ \ z' = \eta \label{euler} \\
	\pti \eta + \nabla'_H \eta\cdot \nabla'_H \phi &=& \pzp \phi \ \ \ \ \ \ \ z' = \eta \label{neum} 
} 
where $F(t)$ is the external forcing. At the lateral and bottom boundaries, we impose the no-flow conditions $\frac{\partial \phi}{\partial \tb{n}} = 0$ while at the interface, we have the Dirichlet condition $\phi|_{z' = \eta} = q$. Here, $\nabla'_H$ represents the horizontal gradient operator $(\partial_{x'} , \partial_{y'},0)$.

These equations can in fact be reformulated as an evolution problem for the surface variables $\eta$ and $q$ only \cite{crsulem,zakh}:
\eqn{
	\pti \eta &=& G[\eta]q. \label{etaeqn}\\
	\pti q &=& -(g-F(t))\eta - \frac{1}{2}|\nabla'_H q|^2 + \frac{(G[\eta]q + \nabla'_H \eta\cdot \nabla'_H q)^2}{2(1 + |\nabla'_H \eta|^2)}. \label{qeqn}
} 
where $G[\eta]q$ is the Dirichlet-Neumann operator (DNO) given by
\eqn{
	G[\eta]q &=& [\nabla' \phi]_{z' = \eta} \cdot(-\nabla'_H \eta,1) = \left[ -\nabla'_H \phi \cdot \nabla'_H \eta + \pzp \phi \right]_{z' = \eta} \label{dno} 
}
and $\phi$ is the solution of (\ref{incomp}) with the boundary conditions specified above.

To sum up, we only need to solve the first-order system (\ref{etaeqn}, \ref{qeqn}) to completely capture the dynamics of the free surface. The problem, of course, lies in the computation of the highly non-local DNO as it requires, in principle, the solution of Laplace's equation on an evolving domain in three dimensions.

\section{The Transformed Field Expansion}
In this section, we develop the Transformed Field Expansion (TFE) method for computing the DNO for Laplace's equation on a cylinder. The key idea of the TFE is to flatten the boundary of the domain and obtain, in place of Laplace's equation, a sequence of associated Poisson equations on the flattened cylinder, the solutions of which yield the potential in the bulk. This reformulation allows us to build a spectrally accurate technique for computing the DNO. The first step is the change of variables
\eqn{
	x' = \rho \cos(\theta), \qquad y' = \rho \sin(\theta), \qquad z' =  \frac{h + \eta}{h}z + \eta \nonumber
}
where $(\rho,\theta)$ are the polar coordinates on a unit disc and $z \in [-h,0]$. In terms of $(\rho,\theta,z)$, therefore, the domain takes the shape of a flat unperturbed cylinder $C$. The metric tensor in the new coordinates is given by $\mathcal{G} = E_1^TE_1$ where
\eqn{
	E_1 = \left(\frac{\partial x^{\prime \hspace{0.05 cm} i}}{\partial x^j}\right) = \begin{pmatrix}
		\cos(\theta) & -\rho \sin(\theta) & 0 \\
		\sin(\theta) & \rho \cos(\theta) & 0 \\
		\left(1 + \frac{z}{h}\right)\eta_{\rho} & \left(1 + \frac{z}{h}\right)\eta_{\theta} & 1 + \frac{\eta}{h}
	\end{pmatrix}. \nonumber
}
In addition, we introduce a new symbol for the bulk potential in the new coordinates
\eqn{u(\rho,\theta,z) = \phi(x',y',z').\label{symbol}}
We next need to determine the transformation that Laplace's equation (\ref{incomp}) undergoes. The Laplace-Beltrami operator applied to both sides of (\ref{symbol}) yields
\eqn{\frac{1}{\sqrt{\det \mathcal{G}}} \frac{\partial}{\partial x^a} \left(\sqrt{\det \mathcal{G}} (\mathcal{G}^{-1})^{ab}u_{x^b}\right) = \Delta \phi = 0. \nonumber
}
Using $\sqrt{\det \mathcal{G}} = (h+\eta)\rho/h$, we obtain 
\eqn{
	\qquad\quad\di\left( (h+\eta)^{-1} EE^T \nabla u\right) = 0, \ \ \mbox{where} \ \ E = \begin{pmatrix}
		h + \eta & 0 & 0 \\
		0 & h+\eta & 0 \\
		-(h+z)\eta_{\rho} & -(h+z)\frac{\eta_{\theta}}{\rho} & h \\
	\end{pmatrix}. \label{diveqn}
}
Here, $\di \ \tb{v} = \rho^{-1}\pr(\rho v_1) + \rho^{-1}\pt v_2 + \pz v_3$ and $\nabla v = (\pr v , \rho^{-1}\pt v, \pz v)^T$ are the divergence and gradient operators in cylindrical coordinates, respectively. Expanding (\ref{diveqn}) leads to
\begin{align*}
(h+\eta)^{-1} \di (EE^T\nabla u) &= -\nabla ((h + \eta)^{-1}) \cdot (EE^T \nabla u) \\
&= (h + \eta)^{-2} (\eta_{\rho} , \rho^{-1}\eta_{\theta} , 0) \ EE^T \ \nabla u.
\end{align*}
Finally, using
\eqn{(h + \eta)^{-1} (\eta_{\rho} , \rho^{-1}\eta_{\theta} , 0) \ E = (\eta_{\rho} , \rho^{-1}\eta_{\theta} , 0)\nonumber}
gives
\eqn{
	\di (EE^T\nabla u) = (\eta_{\rho} , \rho^{-1}\eta_{\theta} , 0) \ E^T  \nabla u. \label{diveqn2}
}

Ostensibly, we have bartered an elementary equation on a challenging domain for a much tougher problem on a simple geometry. This form, however, lends itself to a simplification inspired by boundary perturbation methods. We assume the interface to be a deviation from a flat surface, to wit, 
\eqn{
	\eta(\rho,\theta) &=& \epsilon f(\rho,\theta) \label{pert}
}
for some $\epsilon > 0$. The conditions under which this assumption leads to a useful solution will be made precise later on, but it is worth noting that the actual value of $\epsilon$ is irrelevant. Writing $EE^T = h^2I + \epsilon A_1(f) + \epsilon^2 A_2(f)$ and the first and second columns of $E$ as $B_0+\epsilon B_1(f)$ and $C_0+\epsilon C_1(f)$ in (\ref{diveqn2}) yields
\eqn{
	&&\di \left[(h^2I + \epsilon A_1(f) + \epsilon^2 A_2(f)) \nabla u \right] = \\
	&&\qquad \epsilon f_{\rho} (B_0 + \epsilon B_1(f)) \cdot \nabla u + \epsilon \rho^{-1}f_{\theta} (C_0 + \epsilon C_1(f)) \cdot \nabla u. \nonumber
}
Grouping together similar powers of $\epsilon$ leads to 
\eqn{
	-h^2\Delta u &=& \epsilon\left[\di(A_1(f) \nabla u) - \left(f_{\rho}B_0 + \frac{f_{\theta}}{\rho}C_0\right) \cdot \nabla u\right] +  \label{diveqn3}\\
	&& \qquad \epsilon^2\left[\di(A_2(f) \nabla u) - \left(f_{\rho}B_1(f) + \frac{f_{\theta}}{\rho}C_1(f)\right) \cdot \nabla u\right].  \nonumber
}
The boundary conditions for $u$ meanwhile are
\eqn{
	\qquad u|_{z = 0} = q, \qquad u_z|_{z = -h} = 0, \qquad hu_{\rho}|_{\rho = 1} = \epsilon[-fu_{\rho} + (z + h)f_{\rho}u_z]_{\rho = 1}. \label{bcs}
}

Next, we plug in the field expansion $u(\rho,\theta,z) = \sum_{k = 0}^{\infty} \epsilon^k u_k(\rho,\theta,z)$ in (\ref{diveqn3}) and compare coefficients of like powers of $\epsilon$. This leads to a sequence of Poisson equations for the $u_k$: 
\eqn{
	-h^2\Delta u_k &=& r_k, \label{diveqn4a}
}
where
\eqn{
	r_k &=& \di(A_1 \nabla u_{k-1}) + \di(A_2 \nabla u_{k-2}) \label{diveqn4}\\
	&& \qquad - \left(f_{\rho}B_0 + \frac{f_{\theta}}{\rho}C_0\right)\nabla u_{k-1}  - \left(f_{\rho}B_1 + \frac{f_{\theta}}{\rho}C_1\right)\nabla u_{k-2}. \nonumber
}
The boundary conditions can likewise be obtained from (\ref{bcs}): 
\eqn{
	u_k|_{z = 0} = \delta_{k,0}q, \qquad  \pz u_k|_{z = -h} = 0, \qquad \pr u_k|_{\rho = 1} = \chi_k,  \label{bcs2}
}
where 
\eqn{
	\chi_k &=& h^{-1}\left[-f \pr u_{k-1} + (h + z)(\pr f) \pz u_{k-1}\right]_{\rho = 1}. \label{chik}
}

Observe that the expressions on the right in both (\ref{diveqn4}) and (\ref{bcs2}) depend on lower-order terms in the field expansion. As a result, we can sequentially solve this three-term recurrence for the $u_k$ up to a sufficiently high order $K$, and combine them to obtain an approximation to $u$. It is remarkable that only the previous two terms are needed to compute $u_k$; equations (\ref{diveqn}) and (\ref{diveqn2}) were carefully manipulated to make this happen. Nicholls and Reitich \cite{nicreit04} were the first to discover that a three-term recurrence is possible in the periodic case. Once the recurrence is solved, the potential in the bulk can then be used to compute the Neumann data. Rather than transforming to the original potential $\phi$ and using (\ref{dno}), we express $G[\eta]q$ directly in terms of the $u_k$. In the new coordinates, (\ref{dno}) becomes
\eqn{
	hG[\epsilon f]q &=& h(\pz u) + \epsilon\left(-fG[\epsilon f]q - h\nabla_H f \cdot \nabla_H u\right) \\
	&& \qquad + \epsilon^2 \left(-f \nabla_H f \cdot  \nabla_H u + h |\nabla_H f|^2 \pz u\right) \nonumber
}   
at $z = 0$. Note that we can replace $\nabla_H u$ by $\nabla_H q$ in the above expression. Plugging in $G[\epsilon f]q = \sum_{k = 0}^{\infty} \epsilon^k G_k[f]q$ and the field expansion for $u$ gives
\eqn{
	hG_k[f]q &=& h(u_k)_z - fG_{k-1}[f]q + h|\nabla_H f|^2(u_{k-2})_z  \label{dno3}\\
	&& \quad - \ \delta_{k,1}\left(h \nabla_H f \cdot \nabla_H q\right) - \delta_{k,2}\left(f \nabla_H f \cdot \nabla_H q\right). \nonumber
} 

The formulation (\ref{diveqn4a}) with associated boundary conditions (\ref{bcs2}) is exact and, assuming the expansions converge, yield the true solution to (\ref{diveqn2}). In Section \ref{convproof}, we shall show that, under certain conditions, these expansions do indeed converge strongly.

\section{Solving the Poisson Equations}
In this section, we develop a fast and accurate method to solve the model Poisson equation $-\Delta w = r$ on a flat cylinder $C$ of unit radius and height $h$ with boundary conditions 
\eqn{
	w|_{z = 0} = q, \qquad w_z|_{z = -h} = 0, \qquad w_{\rho}|_{\rho = 1} = \chi. \label{bcs3}
}
Spectral methods for solving similar problems are fairly well developed \cite{boydyu,shen1997efficient,vasil2016tensor}. Broadly speaking, these techniques employ polynomial bases along the radial and vertical axes with a Fourier basis naturally accounting for the azimuthal direction. While our method also follows this general approach, it uses a novel combination of Zernike polynomials on the disc and a Lagrange basis along the $z$-axis. This modal-nodal approach allows for a simpler formulation as well as a well-structured linear system that lends itself to rapid and well-conditioned computations. We now outline our numerical technique, discuss its implementation details, and analyze its computational complexity.

\subsection{The Numerical Method}

The method for solving the model Poisson equation with boundary conditions (\ref{bcs3}) proceeds by searching for an approximate solution in a polynomial subspace. We begin by building a basis for the function space and applying the Galerkin condition. The basis functions are described using a modal representation in $(\rho,\theta)$ and a nodal representation in $z$.

In more detail, for $m,n \in \mathbb{Z}$ with $n \geq 0$, let $P_n^{(0,|m|)}(x)$ be the $n$th $(0,|m|)$ Jacobi polynomial on $[-1,1]$ and set $\mu_{mn} = \sqrt{1 + |m| + 2n}$. Define the functions
\eqn{
	\zeta_{mn}(\rho,\theta) = \mu_{mn}P_n^{(0,|m|)}(2\rho^2 - 1)\rho^{|m|}e^{im\theta} \label{zernike}
}
where $(\rho,\theta)$ are the polar coordinates on the unit disc $D$. These are known as Zernike polynomials \cite{bhatiawolf,boydyu,zernike}. Traditionally, these functions are indexed differently but we prefer this form as it leads to simpler expressions. The weight $\rho^{|m|}$ and the factor $\mu_{mn}$ ensure that the family $\{\zeta_{mn}\}_{m \in \mathbb{Z}, n \geq 0}$ is orthonormal on the unit disc with respect to the inner product
\eqn{
	\ip{v,w}_{L^2(D)} = \frac{1}{\pi} \int_0^{2 \pi} \int_0^1 \overline{v(\rho,\theta)}w(\rho,\theta) \ \rho \ d\rho \  d\theta. \label{l2ip}
}
Indeed, the substitution $\xi = 2\rho^2 - 1$ yields
\eqn{
	% && \hspace*{-32pt}
	&&\ip{\zeta_{m_1n_1} , \zeta_{m_2n_2}}_{L^2(D)}  \nonumber\\
	%	&=& \frac{\mu_{m_1n_1}\mu_{m_2n_2}}{\pi} \int_0^{2\pi} \int_0^1  P^{(0,|m_1|)}_{n_1}(2\rho^2 - 1) P^{(0,|m_2|)}_{n_2}(2\rho^2 - 1) \rho^{|m_1| + |m_2| + 1}e^{i(m_2-m_1)\theta} \ d\rho \ d\theta \nonumber\\
	&& \qquad\qquad = \frac{\mu_{m_1n_1}\mu_{m_2n_2}}{2} \delta_{m_1,m_2} \int_{-1}^1  P^{(0,|m_1|)}_{n_1}(\xi) P^{(0,|m_2|)}_{n_2}(\xi) \left(\frac{1 + \xi}{2}\right)^{\frac{|m_1| + |m_2|}{2} } d\xi. \nonumber
}
%where we used the substitution $x = 2\rho^2 - 1$.
Continuing, we replace the $m_2$'s by $m_1$'s to get
\eqn{
	\ip{\zeta_{m_1n_1} , \zeta_{m_2n_2}}_{L^2(D)} &=& \delta_{m_1,m_2} \frac{\mu_{m_1n_1}\mu_{m_1n_2}}{2^{|m_1| + 1}}  \int_{-1}^1  P^{(0,|m_1|)}_{n_1}(\xi) P^{(0,|m_1|)}_{n_2}(\xi) \left(1 + \xi\right)^{|m_1| }d\xi \nonumber\\
	&=& \delta_{m_1,m_2}\delta_{n_1,n_2} \nonumber
} 
due to the orthogonality of Jacobi polynomials. By the Stone-Weierstrass theorem, the algebra generated by $\{\zeta_{mn}\}$ is dense in $C(D)$, the space of continuous complex-valued functions on $D$. As $C(D)$ in turn is dense in $L^2(D)$, we conclude that $\{\zeta_{mn}\}$ forms an orthonormal basis for $L^2(D)$. 

Let $J$ be a positive integer and let $\{z_j\}_{0 \leq j \leq J}$ be the $(J+1)$ Chebyshev-Lobatto points over $[-h,0]$ defined by 
\begin{flalign*} & & & & &  
&\hspace{0 pt} & z_j = -\frac{h}{2}\left(1 + \cos\left(\frac{\pi j}{J} \right)\right), &\hspace{0 pt} & (0 \leq j \leq J). & 
\end{flalign*}
Also, let $\ell_j$ be the $j$th Lagrange polynomial with respect to these nodes so that $\ell_j(z_i) = \delta_{ij}$.
On $C$, we define the basis functions
\eqn{\psi_{mnj}(\rho,\theta,z) = \zeta_{mn}(\rho,\theta)\ell_j(z),\nonumber}
which are not orthogonal (but nevertheless well-conditioned) with respect to the inner product on $L^2(C)$, namely
\eqn{\ip{v , w} = \frac{1}{\pi} \int_{-h}^{0}\int_0^{2\pi} \int_0^1 \overline{v}  w \ \rho \ d\rho \ d\theta \ dz.\nonumber}

This choice of basis functions allows us to replace the unwieldy Bessel functions and hyperbolic functions along the radial and vertical axes respectively by families of polynomials. These polynomials are easy to evaluate and lend themselves to rapid manipulations. Moreover, as shown in \cite{boydyu}, Zernike polynomials possess distinct advantages in terms of accuracy, cost and storage over other function families on the unit disc. In particular, as we shall also demonstrate, they are much superior to Bessel functions. As a result, the subsequent formulation is considerably simplified and the computations are faster and more accurate.

For positive integers $M,N,J$, we define 
\eqn{
	\mathcal{A}(M,N,J) &=& \{(m,n,j): -M \leq m \leq M , \ 0 \leq n \leq N, \ 0 \leq j \leq J\}. \nonumber
} 
For the model problem $-\Delta w = r$ with boundary conditions (\ref{bcs3}), we begin by imposing the Galerkin condition 
\eqn{\ip{\psi_{m'n'j'} , -\Delta w} = \ip{\psi_{m'n'j'} , r}\nonumber} 
for $(m',n',j')\in \mathcal{A}(M,N,J-1)$. Integrating by parts gives
\eqn{ 
	\iiint_C \overline{\nabla \psi_{m'n'j'}} \cdot \nabla w \ dV  = \iiint_C \overline{\psi_{m'n'j'}} r \ dV + \iint_{\partial C} \overline{\psi_{m'n'j'}} \frac{\partial w}{\partial \tb{n}} \ dA \label{galeqn}
}
where $\partial C = B_u \cup B_d \cup S$; here, $B_u$ and $B_d$ are the upper and lower ends of the cylinder respectively and $S$ is the curved surface. Note that as $\psi_{m'n'j'}|_{B_u} \equiv 0$, there is no contribution from $B_u$. Meanwhile, the second condition in (\ref{bcs3}) ensures that the integral over $B_d$ is also zero. As a result, the boundary contributions can be written as 
\eqn{\iint_{\partial C} \overline{\psi_{m'n'j'}} \frac{\partial w}{\partial \tb{n}} \ dA = \iint_{S} \overline{\psi_{m'n'j'}} w_\rho \ dA = \iint_{S} \overline{\psi_{m'n'j'}} \chi \ dA := I_{m'n'j'}(\chi),\nonumber}
where $\chi$ is the lateral boundary condition in (\ref{bcs3}).

Next, write $w = \sum_{m,n,j } c_{mnj}\psi_{mnj}$ and decompose $r = \sum_{m,n,j} d_{mnj}\psi_{mnj}$, where the sums are over $\mathcal{A}(M,N,J)$ and
\eqn{
	d_{mnj} &=& \ip{\zeta_{mn} , r(\cdot,\cdot,z_j)}_{L^2(D)}, \label{dmnj}
}
to obtain 
\eqn{
	\begin{aligned}
		&\sum_{m,n,j } c_{m,n,j} \iiint_C \overline{\nabla \psi_{m'n'j'}} \cdot \nabla \psi_{mnj} \ dV = \\
		&\qquad\qquad I_{m'n'j'}(\chi) +  \sum_{m,n,j} d_{mnj} \iiint_C \overline{ \psi_{m'n'j'}}  \psi_{mnj} \ dV .
	\end{aligned}\label{gal2}
}
The Dirichlet condition in (\ref{bcs3}) implies that $\{c_{mnj}\}$ is known for $j = J$, so those terms can be moved to the right as well. We therefore have a system of the type $S\vec{c} = T\vec{d} + \vec{\kappa}$, where $S$ and $T$ are the stiffness and mass matrices respectively and $\vec{\kappa}$ is a vector generated by the boundary data $\chi$ and the $\{c_{mnJ}\}$. Observe that $S \in \mathbb{R}^{2MJ(N+1) \times 2MJ(N+1)}$ and $T \in \mathbb{R}^{2MJ(N+1) \times 2M(J+1)(N+1)}$.   

These stiffness and mass integrals can be computed by exploiting the structure of the basis functions. Define the matrices $A(m) \in \mathbb{R}^{(N+1) \times (N+1)}$ for each $m$ and $\hat \Sigma \in \mathbb{R}^{(J+1) \times J}$, $\tilde \Sigma \in  \mathbb{R}^{J \times J}$ by
\eqn{
	A{(m)}_{n'n} &=& \frac{1}{\pi}\int_0^{2\pi}\int_0^1 \overline{\nabla_H \zeta_{mn'}}\cdot \nabla_H \zeta_{mn} \ \rho \ d\rho \ d\theta, \qquad (0 \leq n,n' \leq N),  \nonumber\\
	\hat \Sigma_{jj'} &=& \int_{-h}^{0} \ell_{j}(z) \ell_{j'}(z)  \ dz, \qquad (0 \leq j \leq J, 0 \leq j' \leq J-1), \nonumber\\
	\tilde \Sigma_{jj'} &=& \int_{-h}^{0} \ell'_{j}(z) \ell'_{j'}(z)  \ dz, \qquad (0 \leq j,j' \leq J-1),\nonumber
}
where $\ell_j'= \frac{\partial \ell_j}{\partial z}$. In addition, let $\hat \Sigma$ with the last row omitted be denoted by $\Sigma$. Then,
\eqn{T_{m'n'j',mnj} = \ip{\zeta_{m'n'},\zeta_{mn}}_{L^2(D)}\hat\Sigma_{jj'} = \delta_{m'm}\delta_{n'n}\hat\Sigma_{jj'},\nonumber}
where we used the orthonormality of the $\zeta_{mn}$. In addition,
\eqn
{S_{m'n'j' , mnj} = \delta_{m'm}(A{(m)}_{n'n}\Sigma_{jj'} + \delta_{n'n} \tilde \Sigma_{jj'} ). \label{stiffm}
} 
Finally, for each $m$ define the matrices $\Gamma(m), K(m), G(m) \in \mathbb{C}^{(N+1) \times J}$, and $E(m) \in \mathbb{C}^{(N+1) \times (J+1)}$ by $\Gamma{(m)}_{nj} = c_{mnj}$, $E{(m)}_{nj} = d_{mnj}$ and $K{(m)}_{nj} = \vec{\kappa}_{mnj}$ and $G{(m)} = E{(m)}\hat\Sigma + K{(m)}$ to get 
\eqn{
	A{(m)}\Gamma{(m)}\Sigma + \Gamma{(m)} \tilde \Sigma = G{(m)}. \label{system}
}
Observe that this is a Sylvester equation; alternatively, it can be seen as a sparse linear system of size $J(2M+1)(N+1)$ for the unknown $\{\Gamma(m)\}_{-M \leq m \leq M}$. Instead of using the standard Bartels-Stewart algorithm or an iterative or direct solver, we design an alternative, well-conditioned method for this problem that takes into account the structure of the matrices. First, it can be shown \cite{qadeerthesis} that
\eqn{A{(m)}_{n'n} = 2\mu_{mn'}\mu_{mn}\big[2\gamma_{n'n}(\gamma_{n'n} + |m| + 1) + |m|\big]\nonumber}
where $\gamma_{n'n} = \min\{n',n\}$. The symmetric positive definiteness of $A{(m)}$ allows the eigen-decomposition $A{(m)} = W{(m)} D{(m)}^2 W{(m)}^T$ with $W(m)$ orthogonal.

Next, we devise an efficient method for computing the mass and stiffness matrices $\Sigma$ and $\tilde \Sigma$ appearing in (\ref{system}). In fact, we avoid forming these matrices explicitly and instead construct their Cholesky decompositions directly. As noted in \cite{wilkplasma}, this avoids squaring the condition number of the system. To compute the Cholesky decompositions, let $\{(x_i,\sigma_i)\}$ be the $(J+1)$ point Gauss-Legendre quadrature scheme over $[-h,0]$. Defining the matrices $E_{ij} = \ell_j(x_i)\sigma_i^{1/2}$ and $\tilde E_{ij} = \ell'_j(x_i)\sigma_i^{1/2}$ for $0 \leq j \leq J-1$ allows us to write 
\eqn{\Sigma = E^TE , \qquad \tilde \Sigma = \tilde E^T \tilde E.\nonumber} 
Note that the columns of $E$ (and $\tilde E$) must be linearly independent: any linear combination of the columns that equals zero would correspond to a polynomial of degree at most $J$ ($J-1$ for $\tilde E$) that has $(J+1)$ zeros (at the quadrature points) so that polynomial must be identically zero; the coefficients in the linear combination must necessarily all be zero since the Lagrange polynomials are linearly independent. Thus, the $QR$ factorizations $E = QR$ and $\tilde E = \tilde Q \tilde R$ yield invertible upper triangular matrices. Plugging these decompositions into (\ref{system}) gives  
\eqn{
	(W(m) D(m)^2 W(m)^T) \Gamma(m) (R^TR) + \Gamma(m)(\tilde R^T \tilde R) &=& G(m), \nonumber\\
	D(m)^2 (W(m)^T \Gamma(m)) R^T + (W(m)^T\Gamma(m))\tilde R^T \tilde R R^{-1} &=& W(m)^TG(m)R^{-1}. \nonumber
} 
Next, we compute the singular value decomposition $R\tilde R^{-1} = U\Lambda V^T$ to obtain $\tilde R^T = R^TU\Lambda^{-1}V^T$ and $\tilde R R^{-1} = V\Lambda^{-1} U^T$. This gives $\tilde R^T \tilde R R^{-1} = R^TU\Lambda^{-2}U^T$ and hence
\eqn{\qquad
	\begin{aligned}
		D(m)^2 (W(m)^T \Gamma(m) R^T) + (W(m)^T\Gamma(m) R^T) U\Lambda^{-2}U^T &= W(m)^TG(m)R^{-1}, \\
		D(m)^2 (W(m)^T \Gamma(m) R^T U) + (W(m)^T\Gamma(m) R^T U)\Lambda^{-2} &= W(m)^TG(m)R^{-1}U.
	\end{aligned}\label{simsys}
} 
As both $D(m)^2$ and $\Lambda$ are diagonal, we have
\eqn{(W(m)^T \Gamma(m) R^T U)_{nj} = \frac{(W(m)^TG(m)R^{-1}U)_{nj}}{D(m)^2_{nn} + \Lambda_{jj}},\label{diagsolve}}
which can then be used to solve for $\Gamma(m)$. 

\subsection{Complexity Analysis}
Next, we present a computational analysis of the algorithm described above. The bulk of the computation essentially involves finding the coefficients $\{d_{mnj}\}$ and performing the matrix multiplications specified in (\ref{simsys}). Assuming that $M$ and $N$ are $O(J)$ as well, the latter require $O(J^3)$ operations for each $m$. The former requires the computation of the expressions for $r_k$ in (\ref{diveqn4}) and the projection-interpolant in (\ref{dmnj}). Upon expanding the formulas for $r_k$, we obtain
\eqn{
	r_k(\rho,\theta,z) &=&  -2h^{-1}f \Delta_H u_{k-1} + h^{-1}(h+z)[2 \nabla_H f\cdot \nabla_H(\pz u_{k-1}) + (\pz u_{k-1})\Delta_H f]    \label{rk} \nonumber\\
	&& - h^{-2}f^2\Delta_H u_{k-2} + h^{-2}f(h+z)[2 \nabla_H f\cdot \nabla_H(\pz u_{k-2}) + (\pz u_{k-2})\Delta_H f] \nonumber\\
	&& -h^{-2}(h+z)|\nabla_H f|^2 [2(\pz u_{k-2}) + (h+z)(\pz^2 u_{k-2})].
}  
Note that the solutions $u_{k-1}$ and $u_{k-2}$ are already calculated in terms of the basis functions; in particular, at each horizontal slice indexed by $j$, these solutions are linear combinations of Zernike polynomials. Having found the Zernike modal representation for $f$ as well, carrying out the computation (\ref{dmnj}) comes down to a sequence of projections of the sort
\eqn{&&\mbox{(i)} \ \ \ip{\zeta_{mn} , v_1v_2}_{L^2(D)} \nonumber\\
	&&\mbox{(ii)} \ \ \ip{\zeta_{mn} , \nabla_H v_1 \cdot \nabla_H v_2}_{L^2(D)}, \nonumber \\ % \quad \mbox{and} \nonumber\\
	&&\mbox{(iii)}  \ \ \ip{\zeta_{mn} , (\Delta_H v_1)v_2}_{L^2(D)},\nonumber}
where $v_1$ and $v_2$ are functions on $D$ with known Zernike modal representations. One way to go about this would be to approximate the projections by a pseudo-spectral approach combined with a fast Chebyshev-Jacobi transform \cite{slevinsky2017use}. We instead adopt a Galerkin formulation, perform the integrals exactly, and demonstrate how these can be done efficiently. For instance, for (i), let
\eqn{v_1(\rho,\theta) = \sum_{m,n} \alpha_{mn} \zeta_{mn}(\rho,\theta) , \qquad v_2(\rho,\theta) = \sum_{m,n} \beta_{mn} \zeta_{mn}(\rho,\theta).\nonumber}
Then, the expansion coefficients for $v_1v_2$ are given by
\begin{align*}
\ip{\zeta_{mn} , v_1v_2}_{L^2(D)} &= \\
%\frac{1}{\pi}\sum_{m_1,m_2} \sum_{n_1,n_2} \alpha_{m_1,n_1} \beta_{m_2,n_2}\mu_{m_1n_1}\mu_{m_2n_2}\mu_{mn} \int_0^{2\pi} e^{i(m_1+m_2 - m)\theta} \ d\theta\\
% & \hspace*{30pt} \int_0^1 P_{n_1}^{(0,|m_1|)}(2\rho^2 - 1) P_{n_2}^{(0,|m_2|)}(2\rho^2 - 1) P_{n}^{(0,|m|)}(2\rho^2 - 1) \rho^{|m_1| + |m_2| + |m| + 1} \ d\rho \nonumber\\
& \hspace*{-50pt} \frac{1}{2}\sum_{m_1,m_2} \sum_{n_1,n_2} \alpha_{m_1,n_1} \beta_{m_2,n_2}\mu_{m_1n_1}\mu_{m_2n_2} \mu_{mn} \delta_{m_1+m_2,m} \nonumber\\
& \int_{-1}^1 P_{n_1}^{(0,|m_1|)}(\xi) P_{n_2}^{(0,|m_2|)}(\xi) P_{n}^{(0,|m|)}(\xi) \left(\frac{1 + \xi}{2}\right)^{\frac{|m_1| + |m_2| + |m|}{2}} \ d\xi,
\end{align*}
where $-M \leq m_1,m_2 \leq M$ and $0 \leq n_1,n_2 \leq N$ in the sums above. Observe that the highest degree in the integrands above is $(3N+M)$. Let $N_g = (3N+M)/2 + 1$ and let $\{(\xi_i,\sigma_i)\}$ be the $N_g$-point Gauss-Legendre quadrature scheme on $[-1,1]$. This scheme is guaranteed to correctly integrate all polynomials up to degree $(3N+M+1)$ so it is well-suited for our task. Thus, we have
\eqn{
	\ip{\zeta_{mn} , g_1g_2}_{L^2(D)} &=& \frac{1}{2}\sum_{m_1,m_2 } \sum_{n_1,n_2}  \alpha_{m_1,n_1} \beta_{m_2,n_2}\mu_{m_1n_1}\mu_{m_2n_2} \mu_{mn} \delta_{m_1+m_2,m} \nonumber\\
	&& \; \sum_{i = 1}^{N_g} P_{n_1}^{(0,|m_1|)}(\xi_i) P_{n_2}^{(0,|m_2|)}(\xi_i) P_{n}^{(0,|m|)}(\xi_i) \left(\frac{1 + \xi_i}{2}\right)^{\frac{|m_1| + |m_2| + |m|}{2}} \sigma_i. \nonumber
}
This quadrature rule is preferred to a Gauss-Jacobi scheme because it allows us to write all expressions of this type as 
\begin{align}
& \frac{1}{2}\sum_{i = 1}^{N_g} \left\{ \sum_{m_1,m_2} \delta_{m_1+m_2,m} \left[\sum_{n_1 = 0}^N \alpha_{m_1,n_1}\mu_{m_1n_1}   P_{n_1}^{(0,|m_1|)}(\xi_i)\left(\frac{1 + \xi_i}{2}\right)^{\frac{|m_1|}{2}} \sigma_i^{1/3}\right] \right.  \nonumber\\
& \hspace*{.5in} \left.\left[\sum_{n_2 = 0}^N \beta_{m_2,n_2}\mu_{m_2n_2} P_{n_2}^{(0,|m_2|)}(\xi_i)\left(\frac{1 + \xi_i}{2}\right)^{\frac{|m_2|}{2}} \sigma_i^{1/3}\right] \right\} \times \\
& \hspace*{2.3in} \times \mu_{mn} P_{n}^{(0,|m|)}(\xi_i) \left(\frac{1 + \xi_i}{2}\right)^{\frac{|m|}{2}} \sigma_i^{1/3}. \nonumber
\end{align}
The terms of the form $\mu_{mn} P_{n}^{(0,|m|)}(\xi_j) \left(\frac{1 + \xi_j}{2}\right)^{\frac{|m|}{2}} \sigma_j^{1/3}$ can be pre-computed for all the polynomials and quadrature points and weights. Given $\{\alpha_{mn}\}$ and $\{\beta_{mn}\}$, the sums in the square parentheses are evaluated at the quadrature points and their Fast-Fourier Transforms computed; this requires $O(M(N + \log(M)))$ operations for each $i$. Multiplying them together and taking the inverse FFTs executes the convolution inside the curly braces and requires an additional $O(M\log(M))$ operations. Finally, we can multiply the external factors at a cost of $O(MN)$ and sum over the quadrature index $i$ to obtain all the modal coefficients of $v_1v_2$. The complexity of computing a projection of type (i) therefore comes up to $O(M(M+N)(N+\log(M)))$. Types (ii) and (iii) differ from (i) only in that they involve the derivatives of the functions; due to the derivative expressions for Jacobi polynomials, this only requires appropriate permutations of the coefficients and thus these types also possess the same complexity (see \cite{qadeerthesis} for details). If $M,N = O(J)$, the procedures require $O(J^3)$ operations. The entire projection-interpolant for $r_k$ in (\ref{dmnj}) requires these computations for each horizontal slice indexed by $j$ and, hence, can be carried out in $O(J^4)$ steps. We conclude that the total complexity of our Poisson solver on a flat cylinder is $O(J^4)$.

\section{Convergence Proof}\label{convproof}
In this section, we analyze our basis functions in detail and use them to develop a convergence proof for the method outlined above. Along the way, we shall also establish the superiority of Zernike polynomials over Bessel functions for representing smooth functions on the unit disc.

As before, let $D$ be the open unit disc in the plane. For integral values of $s \geq 0$, let $H^s(D)$ denote the usual $L^2$-Sobolev space on $D$. Similarly, let $H_{\sigma}^s((-h,0))$ be the Sobolev space equipped with the norm
\eqn{\norm{v}_{H_{\sigma}^s((-h,0))}^2 = \int_{-h}^{0} \sum_{k = 0}^{s} |v^{(k)}(z)|^2 \sigma(z) \ dz ,\nonumber}
where ${\sigma}(z) = \frac{1}{2}\left(\frac{-z}{h}\left(1 + \frac{z}{h}\right)\right)^{-1/2}$. Note that under the transformation $x = 1 + 2z/h$, the weight function gets changed to the Chebyshev weight $(1 - x^2)^{-1/2}$ over $[-1,1]$.     

Recall the Zernike polynomials $\{\zeta_{mn}\}$ defined on $D$ in (\ref{zernike}). We next present an alternate, sharper, characterization for these functions that will lead to a useful approximation result that shall feature prominently in the analysis of our algorithm.
\begin{lem}\label{eigenf}
	Define the linear operator
	\eqn{Lu = -\rho^{-1}\pr\left[\rho(1 - \rho^2)\pr u\right] - \rho^{-2}\pt^2u.\nonumber}
	Then, 
	\begin{enumerate}
		\item[(a)] $L$ is bounded from $H^{l+2}(D)$ to $H^l(D)$ for any integer $l \geq 0$.
		\item[(b)] $\{\zeta_{mn}\}$ are eigenfunctions of $L$ with eigenvalues $\lambda_{mn} = (|m| + 2n)(|m| + 2n + 2)$. 
	\end{enumerate}
\end{lem}
\begin{proof}
	\begin{enumerate}
		\item[(a)] This follows easily from rewriting 
		\eqn{Lu = -\Delta u + (\rho^2\pr^2 + 3\rho \pr)u\nonumber}
		and using the fact that both operators above are bounded from $H^{l+2}(D)$ to $H^l(D)$ for any integer $l \geq 0$.
		
		\item[(b)] See \cite{dunkl2014orthogonal,matsushima1995spectral}.
	\end{enumerate}
	
\end{proof}

Observe that the operator $L$ defined above is self-adjoint in $L^2(D)$. This fact has a crucial bearing on the next approximation result. For $v \in H^s(D)$, define the projection \eqn{\mathcal{P}_{MN}v(\rho,\theta) = \sum_{|m|\leq M} \sum_{0 \leq n \leq N} a_{mn}\zeta_{mn}(\rho,\theta)\nonumber}
where $a_{mn} = \ip{\zeta_{mn},v}_{L^2(D)}$. The Stone-Weierstrass theorem and a standard density argument show that the $\{\zeta_{mn}\}$ form a basis for $L^2(D)$ and, as a result, we have $\lim_{M,N \to \infty}\mathcal{P}_{MN}v = v$ in the $L^2$ sense. The next theorem provides a precise estimate for the approximation error. The definition of $H^s(D)$ for real $s > 0$ and the interpolation theorem used in the following proof are stated in Appendix \ref{appa}. 

\begin{thm}\label{discerr}
	Let $s$ be a positive real number and let $v \in H^s(D)$. For $M,N \geq 0$, let $\mathcal{P}_{MN}v$ be the projection of $v$ on $\{\zeta_{mn}\}$ as described above. Then, there exists a constant $P_s$ such that
	\eqn{\norm{v - \mathcal{P}_{MN}v}_{L^2(D)} \leq P_s \min(M,2N)^{-s}\norm{v}_{H^s(D)}.\nonumber}
\end{thm}
\begin{proof}
	We follow the standard argument presented in \cite{bmaday97}. First suppose that $s = 2k$ for some integer $k \geq 1$. Note that $v - \mathcal{P}_{MN}v = \sum_{m,n \in \Lambda_{MN}} a_{mn}\zeta_{mn}$ where $\Lambda_{MN} = \{(m,n): |m| > M \mbox{ or } n > N\}$. From Lemma \ref{eigenf}(b), we have
	\eqn{
		a_{mn} &=& \ip{\zeta_{mn},v} = \lambda_{mn}^{-k}\ip{L^k\zeta_{mn},v} =  \lambda_{mn}^{-k}\ip{\zeta_{mn},L^kv}. \nonumber
	}
	It follows that
	\eqn{
		\norm{v - \mathcal{P}_{MN}v}_{L^2(D)}^2 &=& \sum_{m,n \in \Lambda_{MN}} |a_{mn}|^2 \nonumber\\
		&=& \sum_{m,n \in \Lambda_{MN}} \lambda_{mn}^{-2k}|\ip{\zeta_{mn} , L^kv}|^2 \nonumber\\
		&\leq& \min(M,2N)^{-4k} \sum_{m,n \in \Lambda_{MN}} |\ip{\zeta_{mn} , L^kv}|^2 \nonumber\\
		&\leq& \min(M,2N)^{-4k} \norm{L^kv}_{L^2(D)}^2 \nonumber
	} 
	where we used the fact that $\lambda_{mn} = (|m| + 2n)(|m| + 2n + 2) \geq (\min(M,2N))^2$ for $m,n \in \Lambda_{MN}$. From Lemma \ref{eigenf}(a), we have
	\eqn{\norm{L^kv}_{L^2(D)} \leq P_{2k}\norm{v}_{H^{2k}(D)}\nonumber}
	so we have the result in the case that $s = 2k$. 
	
	Next, let $s$ be a positive real number that is not an even integer and choose an integer $k$ such that $s = 2(k + \nu)$ for $0 < \nu < 1$. We have established that the operator $(I - \mathcal{P}_{MN})$ is continuous from $H^{2k}(D)$ to $L^2(D)$ with norm $P_{2k}\min(M,2N)^{-2k}$ and from $H^{2k+2}(D)$ to $L^2(D)$ with norm $P_{2k+2}\min(M,2N)^{-(2k+2)}$. Interpolating between these (see Appendix \ref{appa}), we deduce that it is bounded from $H^{s}(D)$ to $L^2(D)$ with norm bounded by $P_s\min(M,2N)^{-s}$, where $P_s = P_{2k}^{1-\nu}P_{2k+2}^{\nu}$.   
\end{proof}

Theorem \ref{discerr} shows that the rate of error decay is faster than any power of $\min(M,2N)^{-1}$. This is commonly termed spectral accuracy \cite{canuto06}. Also note that if $v$ has a finite highest angular frequency $m'$ so that $\ip{\zeta_{mn},v}_{L^2(D)} = 0$ for $m > m'$, then, by the same argument as above, the error decay occurs at rate $N^{-s}$, provided $M \geq m'$.

Next, we introduce a standard Chebyshev interpolation result that will allow us to study the approximation properties of the basis functions $\{\psi_{mnj}\}$ on $C$. Let $u \in H_{\sigma}^s((-h,0))$ and let $\{z_j\}_{0 \leq j \leq J}$ be the Chebyshev-Lobatto nodes on $(-h,0)$. Let $\ell_j$ be the $j$th Lagrange interpolating polynomial on these nodes and set 
\eqn{u_J(z) = \sum_{j = 0}^J u(z_j)\ell_j(z),\nonumber} 
that is, the $J$th Chebyshev interpolant for $u$. We then have the following result (Statement 5.5.22 from \cite{canuto06}).
\begin{lem}\label{zappr}
	Let $s , J\geq 0$ be integers. Take $u \in H_{\sigma}^s((-h,0))$ and let $u_J$ be the $J$th Chebyshev interpolant for $u$. Then, there exists a constant $Q_s$ such that
	\eqn{\norm{u - u_J}_{L_{\sigma}^2((-h,0))} \leq Q_sJ^{-s}\norm{u}_{H_{\sigma}^s((-h,0))}\nonumber} 
\end{lem}

Recall that $C = D \times (-h,0)$ is the flat cylinder. For $w \in H^s_{\sigma}(C) = H^s(D) \otimes H^s_{\sigma}((-h,0))$, define the projection-interpolant $w_{MNJ}$ by
\eqn{w_{MNJ}(\rho,\theta,z) = \sum_{0 \leq j \leq J} \mathcal{P}_{MN}w(\rho,\theta,z_j) \ell_j(z).\nonumber}
We next combine the approximation estimate Theorem \ref{discerr} for Zernike polynomials on $D$ and Lemma \ref{zappr} along the $z$-axis to obtain approximation estimates for the projection-interpolant on the entire cylinder. 
\begin{thm}\label{cylapp}
	Let $s , M,N,J \geq 0$ be integers. Let $w \in H^s(C)$ and let $w_{MNJ}$ be the corresponding projection-interpolant. Then, there exists a constant $R_s$ such that
	\eqn{\norm{w - w_{MNJ}}_{L^2_{\sigma}(C)} \leq R_s(\min(M,2N)^{-s} + J^{-s})\norm{v}_{H^s_{\sigma}(C)}. \nonumber}
\end{thm}
\begin{proof}
	Observe that
	%  \eqn{
	\begin{equation}
	\begin{aligned}
	\norm{w - w_{MNJ}}_{L^2_{\sigma}(C)} &\leq \norm{w - \mathcal{P}_{MN}w}_{L^2(D) \otimes L_{\sigma}^2((-h,0))} \\
	&\qquad + \norm{\mathcal{P}_{MN}w - w_{MNJ}}_{L^2(D) \otimes L_{\sigma}^2((-h,0))}
	\end{aligned}  \label{trineq}
	\end{equation}
	We have, by Theorem \ref{discerr},
	\eqn{
		\norm{ w - \mathcal{P}_{MN}w}_{{L^2(D)} \otimes L_{\sigma}^2((-h,0))} &\leq& P_s\min(M,2N)^{-s}\norm{w}_{H^s(D) \otimes L_{\sigma}^2((-h,0))} \nonumber\\
		&\leq&  P_s\min(M,2N)^{-s}\norm{w}_{H_{\sigma}^s(D) \otimes H^s((-h,0))} \nonumber
	}
	and, by Lemma \ref{zappr},
	\eqn{
		\norm{\mathcal{P}_{MN}w - w_{MNJ}}_{L_{\sigma}^2((-h,0))\otimes L^2(D)} &\leq& Q_sJ^{-s} \norm{\mathcal{P}_{MN}w}_{L^2(D) \otimes H_{\sigma}^s((-h,0))} \nonumber\\
		&\leq& Q_sJ^{-s}\norm{w}_{H^s(D) \otimes H_{\sigma}^s((-h,0))}. \nonumber
	}
	Putting these together in (\ref{trineq}) and setting $R_s = \max\{P_s,Q_s\}$ gives the desired result.
	
\end{proof}

The approximation estimate can be used to yield a convergence proof for our computational method. We refer the reader to \cite{nicshen09} for details of the proof. We first have the analyticity result for the transformed field expansion (Theorem 3.1 of \cite{nicshen09})
\begin{thm}\label{analytic}
	Given an integer $s \geq 1$, if $q \in H^{s + 3/2}(D)$ and $f \in H^{s + 2}(D)$, then there exist constants $E_1,E_2 > 0$ such that
	\eqn{\norm{u_k}_{H^{s+2}_{\sigma}(C)} \leq E_1 \norm{q}_{H^{s+3/2}(D)}B^k\nonumber}
	for any constant $B \geq E_2\norm{f}_{H^{s+2}(D)}$.  
\end{thm}
This result demonstrates that the transformed field expansion $\sum_{k = 0}^{\infty} \epsilon^k u_k$ converges for $B\epsilon < 1$. As a result, the technique is guaranteed to yield the exact solution $u$ of (\ref{diveqn}). We only need to show that our numerical solution converges in an appropriate sense to $u$.

Let $u^k_{MNJ}$ be the solutions to the Poisson problems (\ref{diveqn4a}) obtained from the spectral method. In addition, let 
\eqn{u_{KMNJ} = \sum_{k = 0}^K \epsilon^k u^k_{MNJ}\nonumber}
denote the numerical approximation to $u$.  We then have the following convergence result. 
\begin{thm}\label{uconv}
	Assume $f \in H^s(D)$ and $q \in H^{s-1/2}(D)$ for some integer $s \geq 3$. Then
	\begin{equation}
	\norm{u - u_{KMNJ}}_{L_{\sigma}^2(C)} \leq (B\epsilon)^{K+1} + R_s(\min(M,2N)^{-s} + J^{-s})\norm{q}_{H^{s-1/2}(D)} \label{uconvthm}
	\end{equation}
	for any constant $B \geq E_2\norm{f}_{H^s(D)}$ such that $B\epsilon < 1$, where $E_2$ is the constant from Lemma \ref{analytic}, and $R_s$ is the constant from Theorem \ref{cylapp}.
\end{thm}
The proof combines the analyticity result from Theorem \ref{analytic} and the approximation estimate on the cylinder from Theorem \ref{cylapp}. See the proof of Theorem 2.1 of \cite{nicshen09} for details on how to combine these.

\section{Numerical Results}

We first numerically confirm the spectral accuracy of the modal representation. For $k \geq 0$ and any $\alpha > 0$, let \eqn{f_k(\rho,\theta) = e^{-\alpha \rho^2} \rho^k \cos(k\theta). \nonumber}
The coefficients $\ip{\zeta_{mn},f}_{L^2(D)}$ in the Zernike representation of $f$ can be computed by using a high-order Gauss-Jacobi quadrature rule. Theorem \ref{discerr} predicts that the error $\norm{f - \mathcal{P}_{MN}f}_{L^2(D)}$ will decay faster than any power of $N^{-1}$, provided that $M \geq k$. Figure \ref{zererrs:1} confirms the spectral decay for multiple values of $k$ and $\alpha$ with $1 \leq N \leq 30$ and $M = 16$. In order to show that this representation avoids spurious behavior, we have shown the $L^{\infty}$ errors. These were computed by sampling the functions on a fine mesh consisting of 14230 points. That the $L^2$ errors behave similarly follows from this since the domain is bounded.

\def \sclo {0.54}
\def \sclt {0.418}
\def \sclta {0.33}

\begin{figure}
	\centering
	\subfigure[$L^{\infty}$ errors for smooth functions]
	{\includegraphics[scale=\sclta,trim=0 0 40 0,clip]{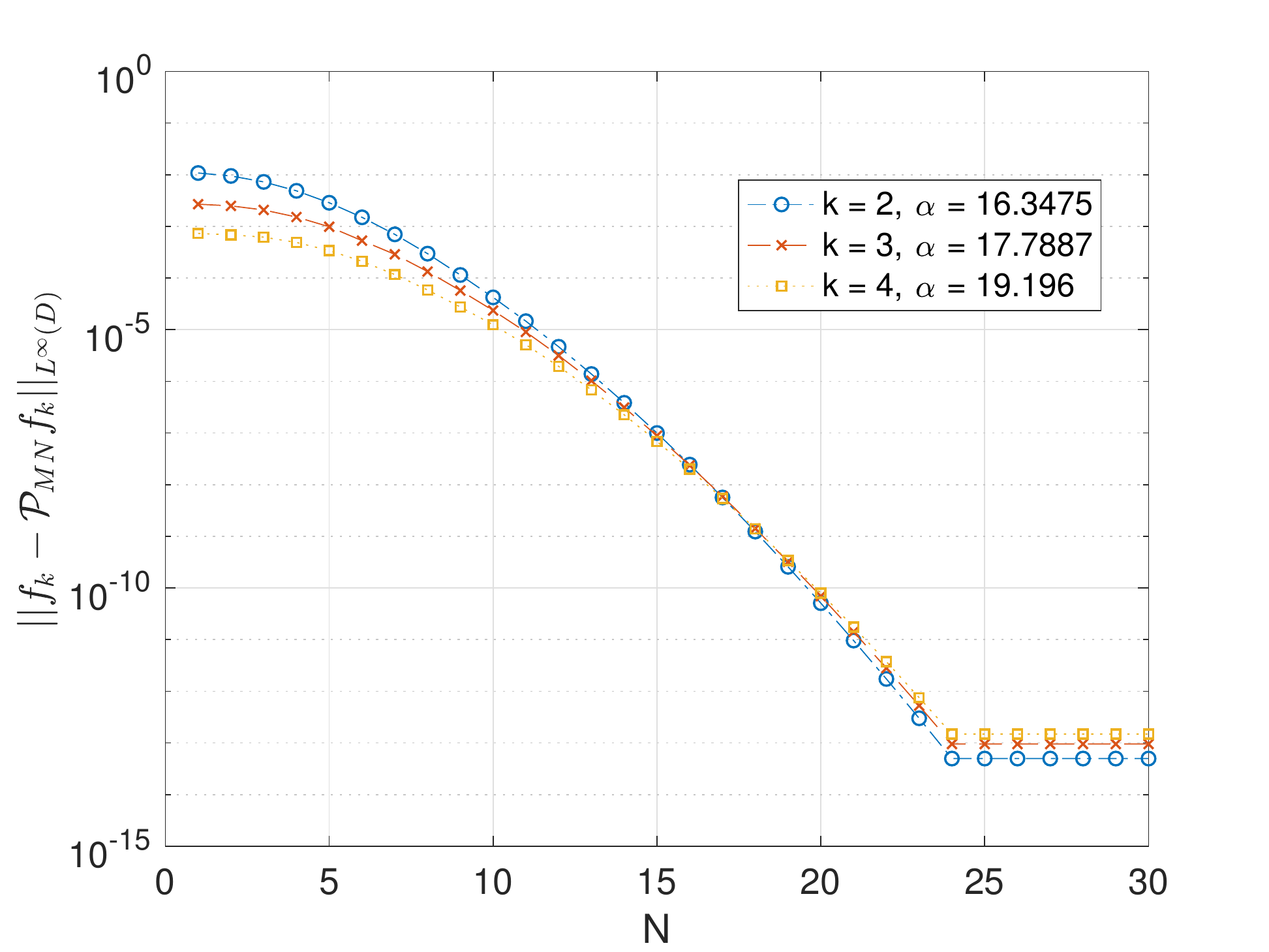}
		\label{zererrs:1}
	}
	\subfigure[$L^{2}$ errors for rough functions]
	{\includegraphics[scale=\sclta,trim=0 0 40 0,clip]{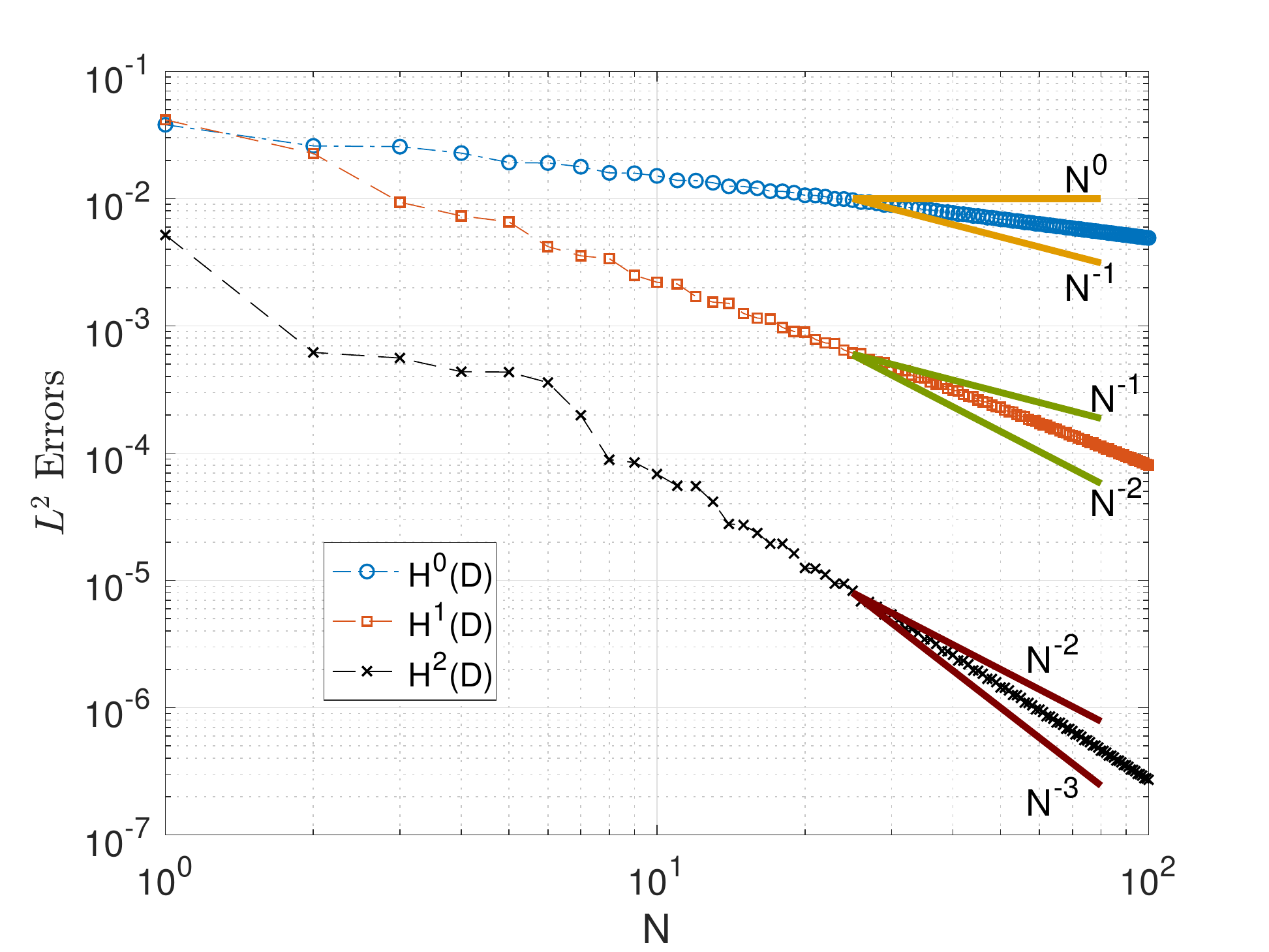}
		\label{zererrs:2}
	}
	\caption{(a) $L^\infty$ errors in the Zernike representation of $f_k(\rho,\theta) = e^{-\alpha \rho^2}\rho^k \cos(k \theta)$ vs. $N$. Observe that the errors decay super-algebraically (in fact exponentially) in $N$, as established in Theorem \ref{discerr}, since $f$ is infinitely differentiable, and that the same holds for pointwise errors. 
		(b) $L^2$ error plots for Zernike representations of functions that belong to $H^0(D)$, $H^1(D)$ and $H^2(D)$. The convergence is algebraic and controlled by $N^{-s}$ for an $H^s$ function. 
	}\label{zererrs}
\end{figure}

As further validation of Theorem \ref{discerr}, in Figure \ref{zererrs:2} we show the $L^2$ error plots for the Zernike representations of non-smooth functions. We consider, in turn, functions whose radial components have a jump discontinuity, a cusp and a discontinuous second derivative. As a result, these belong to the spaces $H^0(D)$, $H^1(D)$ and $H^2(D)$ respectively. The error decay plots are in agreement with Theorem \ref{discerr}: the convergence rate is at least $N^{-s}$ for a function belonging to $H^s(D)$.

Next, we reiterate the advantages of Zernike polynomials over Bessel functions for representational purposes. First, let $J_m$ be the $m$th Bessel function of order zero and $a_{mn}$ the $n$th positive zero of $J_m'(\cdot)$. It follows from the orthogonality of $\{J_m(a_{mn}\rho)\}_{n \geq 1}$ that any square integrable function $g$ on the unit disc can be represented in $L^2(D)$ as
\eqn{g(\rho,\theta) = \sum_{m \geq 0, n \geq 1} \beta_{mn}J_{m}(a_{mn}\rho)e^{im\theta},\nonumber}
where 
\eqn{\beta_{mn} = \left[2\pi  \int_0^1 (J_m(a_{mn}\rho))^2 \ \rho \ d\rho \right]^{-1} \int_0^{2\pi}\int_0^1 g(\rho,\theta) J_m(a_{mn}\rho) e^{-im\theta} \ \rho \ d\rho \ d\theta.\nonumber}
In order to compare the two representational techniques, we represent functions from one family in terms of the other and vice versa. More precisely, for testing the Bessel representation, consider
\eqn{
	g_{m'n'}(\rho,\theta) = \zeta_{m'n'}(\rho,\theta) - h_{m'}(\rho,\theta) \label{zern}
}
where
\eqn{
	h_{m'}(\rho,\theta) = \left\{\begin{matrix}
		\left(\frac{2n'(m'+n'+1)}{m'} + 1\right)\sqrt{\frac{1+m'+2n'}{1+m'}}\zeta_{m'0}(\rho,\theta) & , & m' > 0 \\
		\left(\frac{2n'(n'+1)}{4}\right)\sqrt{\frac{1+2n'}{3}}\zeta_{01} (\rho,\theta)& , & m' = 0
	\end{matrix} \right.
}
for any $(m',n')$. The corrections $h_{m'}$ ensure that $\pr g_{m'n'}|_{\rho = 1} = 0$, in agreement with the Bessel functions used in the representation. Figure \ref{jacbescom} displays the results for the $L^{\infty}$ norm. The plots show that the error decay for the Bessel function representation is algebraic, as opposed to the spectral accuracy possessed by Zernike polynomials. An intuitive reason for this is that Bessel functions are not as oscillatory as Zernike polynomials near $\rho = 1$ and, as a result, are less accurate close to the outer boundary. A useful analog is the comparison of a Fourier sine-series on $[0,\pi]$ with a Chebyshev expansion. The zeros of the latter cluster near the boundaries like $1/n^2$, where $n$ is the mode number, and yield spectrally accurate representations. Meanwhile, the zeros of the former cluster like $1/n$ and lead to algebraic decay of mode amplitudes. More concretely, as established in \cite{boydyu}, the error decay for the Bessel representation of a function $g$ is controlled by the highest integer $p \geq 0$ for which $\pr \Delta^k g|_{\rho = 1} = 0$ for $0 \leq k \leq p-1$; in this case, the asymptotic rate is $O(N^{-(2p+1/2)})$. Since this condition is unlikely to hold for all integer values of $p$, a spectral decay rate is seldom exhibited. On the other hand, no such condition is required for Zernike polynomials, as proven in Theorem \ref{discerr} and illustrated in Figure \ref{zererrs}. As a result, a super-algebraic rate of convergence is obtained for all smooth functions while an algebraic rate only shows up for non-smooth functions.

\begin{figure}
	\centering
	\subfigure[Bessel expansion of $g_{m'n'}(\rho,\theta)$]
	{\includegraphics[scale=\sclta,trim=0 0 40 0,clip]{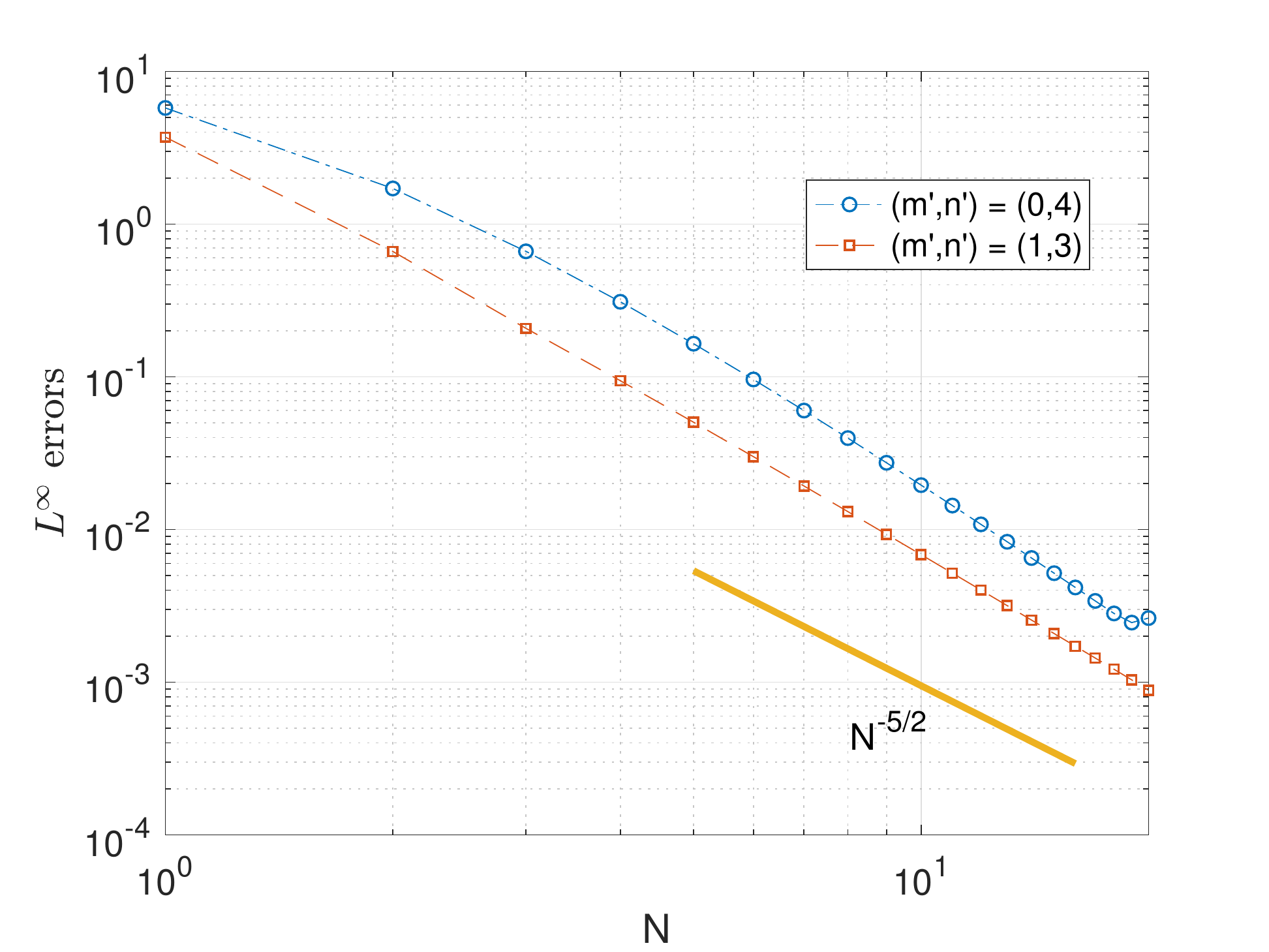}
		\label{jacbescom:1}
	}
	\subfigure[Zernike expansion of $J_{m'}(a_{m'n'\rho})\cos(m'\theta)$]
	{\includegraphics[scale=\sclta,trim=0 0 40 0,clip]{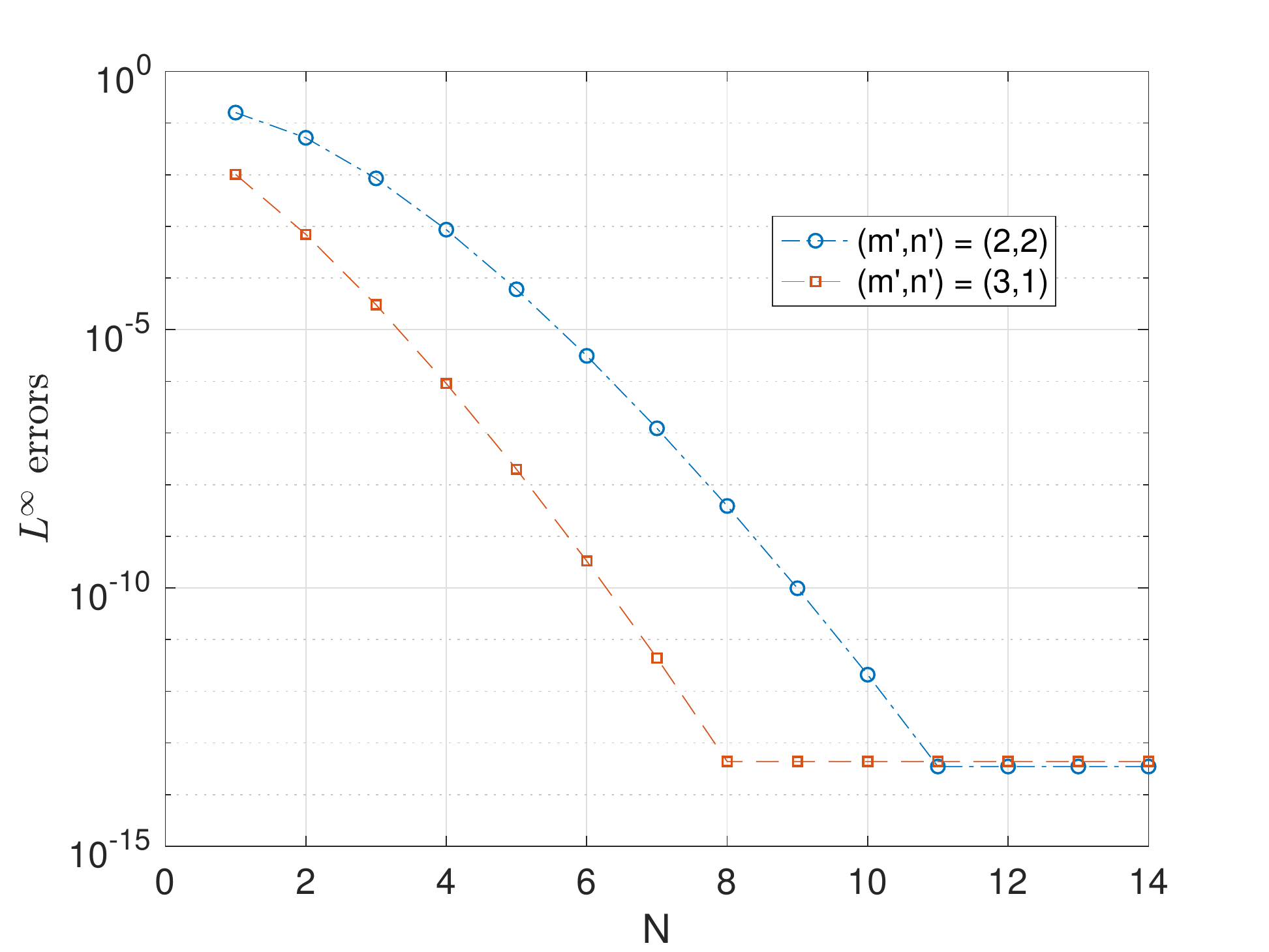}
		\label{jacbescom:2}
	}
	\caption{(a) $L^{\infty}$ errors in the representation of the corrected Zernike polynomial $g_{m'n'}(\rho,\theta)$ in terms of Bessel functions decay algebraically. Observe that the decay rate is $~ N^{-5/2}$ which is in agreement with the expected $N^{-(2p+1/2)}$ with $p = 1$ since only the first compatibility condition is satisfied by $g_{m'n'}$. (b), the $L^{\infty}$ errors in the representation of $J_{m'n'}(a_{m'n'}\rho)e^{im'\theta}$ in terms of Zernike polynomials decay exponentially. }\label{jacbescom}
\end{figure}

\begin{figure}[tbph]
	\centering
	\scalebox{\sclo}
	{\includegraphics{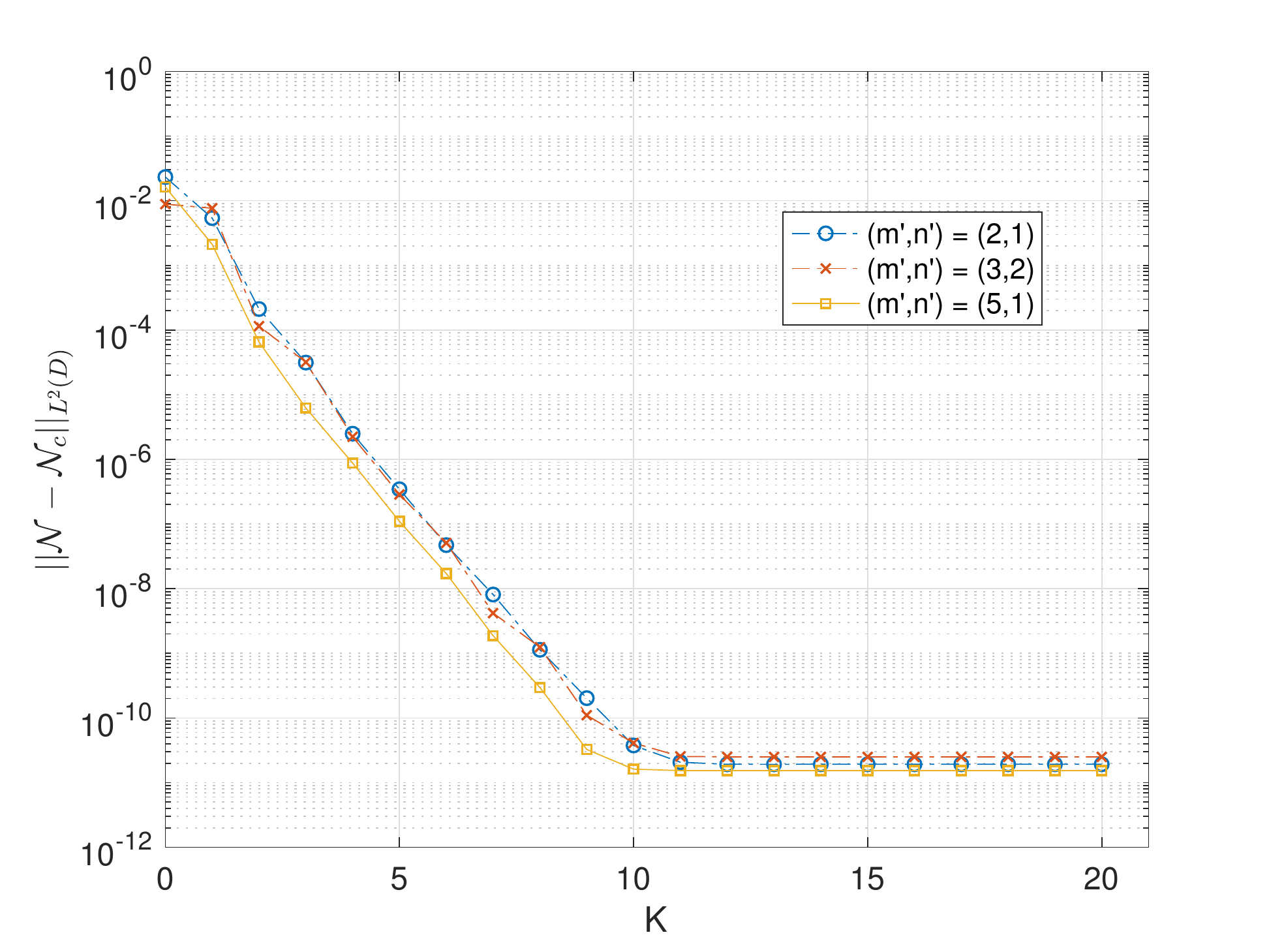}}
	\caption{Convergence of Neumann data vs TFE order $K$ for various Dirichlet conditions. The parameter choices are $M = 32$, $J = 20$, $N = 42$, $h = 1.0$ and $\epsilon = 0.2$.}\label{erpl}
\end{figure}

In order to test the DNO algorithm, we consider a case where Laplace's equation can be analytically solved and we have a closed form for the Neumann data. Let $(\rho,\theta,z')$ be the cylindrical coordinates for the unflattened cylinder (so the interface is $z' = \eta$). The general solution of (\ref{incomp}) on a cylinder with no-flow boundary conditions on the lateral and bottom walls is
\eqn
{\phi(\rho,\theta,z') = \sum_{m \in \mathbb{Z}, n \geq 1} b_{m,n}J_{|m|}(a_{|m|n}\rho)e^{im\theta}\cosh(a_{|m|n}(z'+h)), \label{gensol}
}
where $J_m$ and $a_{mn}$ are as defined earlier. Note that as $\phi$ is real-valued, we must have $b_{-m,n} = \overline{b_{m,n}}$ for all $m,n$. Fix $m' \geq 0, n' > 0$ and suppose that for a given interface $\eta(\rho,\theta)$, the Dirichlet data is of the form
\eqn{
	q(\rho,\theta) &=& J_{m'}(a_{m'n'}\rho)\cos(m'\theta)\frac{\cosh(a_{m'n'}(\eta(\rho,\theta) + h))}{\cosh(a_{m'n'}\norm{\eta+h}_{\infty})}. \label{eq:q:exa}
}
The particular solution of (\ref{gensol}) is then
\eqn{\phi(\rho,\theta,z') = J_{m'}(a_{m'n'}\rho)\cos(m'\theta)\frac{\cosh(a_{m'n'}(z'+h))}{\cosh(a_{m'n'}\norm{\eta+h}_{\infty})},\nonumber}
which can be used to compute the the Neumann data $\mathcal{N}(\rho,\theta)$ explicitly by (\ref{dno}). Figure \ref{erpl} displays the decay in the $L^2$ errors in the computed Neumann data $\mathcal{N}_c(\rho,\theta)$  for $f(\rho,\theta) =  J_{1}(a_{11}\rho)\cos(\theta)$
and $(m',n') = (2,1)$, $(3,2)$ and $(5,1)$.

\begin{figure}[tbph]
	\centering
	\scalebox{\sclo}
	{\includegraphics{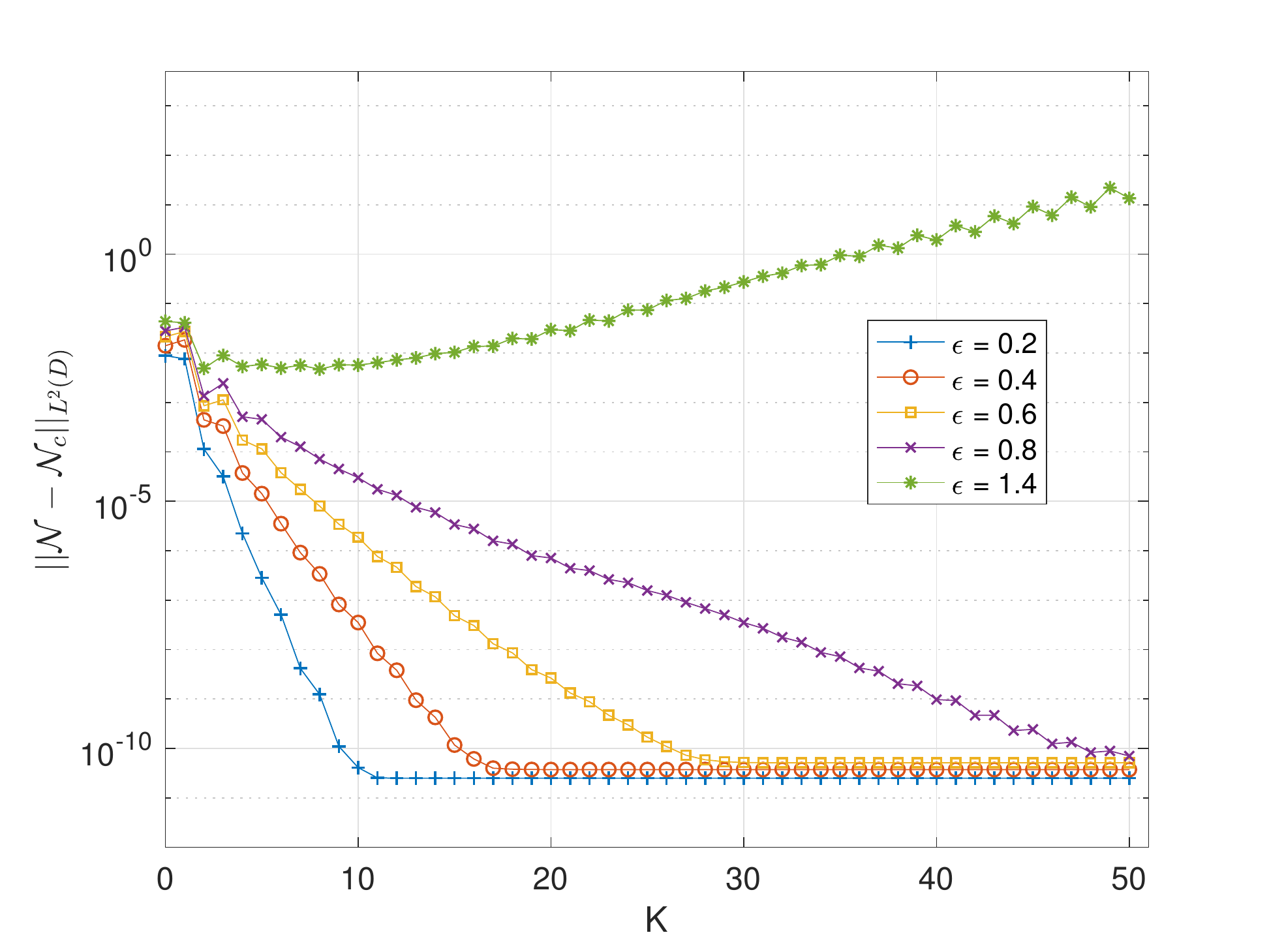}}
	\caption{Errors in Neumann data vs. TFE order $K$ for different values of $\epsilon$, for $(m',n')=(3,2)$ in Fig. \ref{erpl}. Observe that the error decay, while still exponential, slows down as $\epsilon$ is increased. For too large a value, the method fails to converge at all.}\label{erpleps}
\end{figure}

\begin{figure}[tbph]
	\centering
	\includegraphics[width=\linewidth]{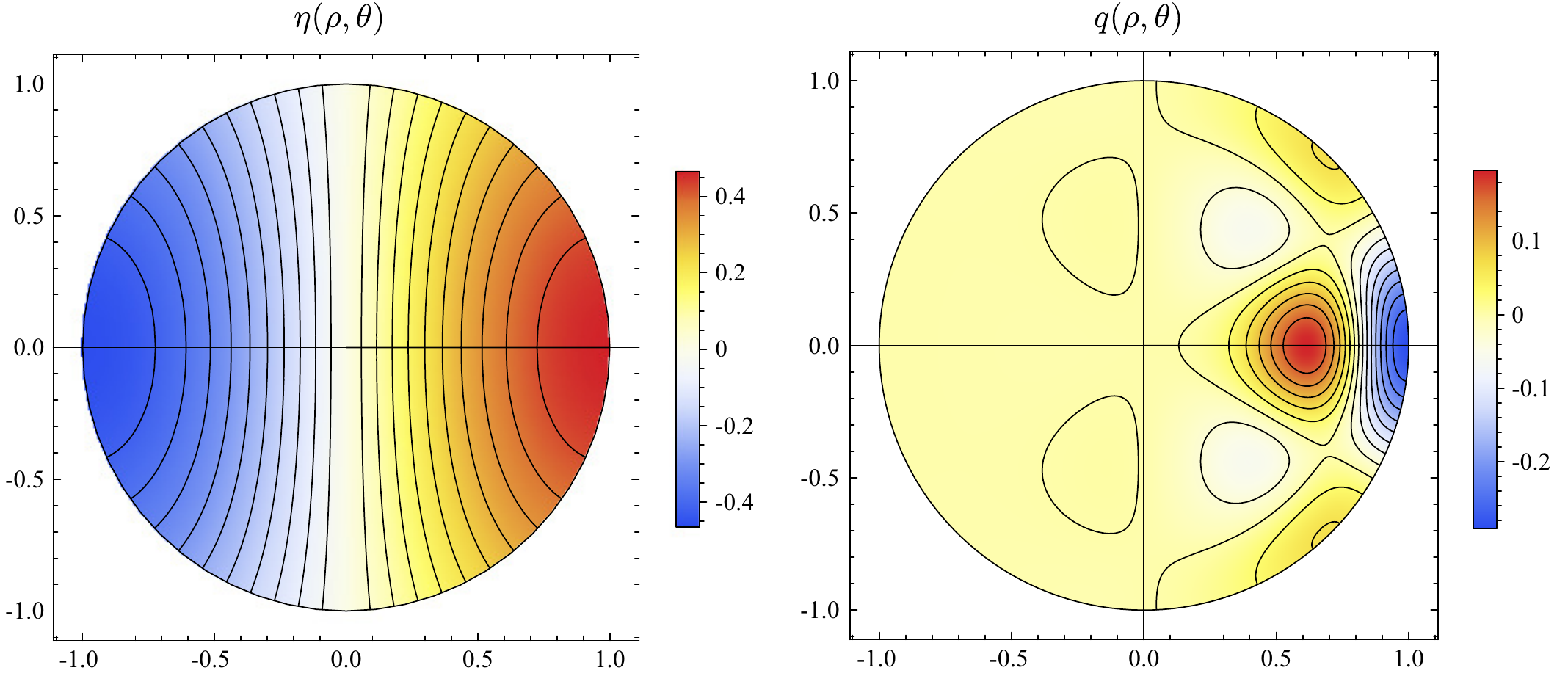}
	\caption{Contour plots of $\eta(\rho,\theta)$ and $q(\rho,\theta)$ corresponding to $\epsilon=0.8$ in Figure~\ref{erpleps}.
		Here $\eta(\rho,\theta)=0.8J_1(a_{11}\rho)\cos\theta$ and $q(\rho,\theta)$ is given by (\ref{eq:q:exa})
		with $(m',n')=(3,2)$. The TFE recursion converges in spite of the large deviation of $\eta(\rho,\theta)$ from
		the flat state.}\label{fig:contourPlotsE8}
\end{figure}

An indication of the role played by the size of the interface $\eta$
can be garnered by comparing the rates of convergence for different
values of $\epsilon$. Figure \ref{erpleps} shows that the rate slows
down for larger values, consistent with Theorem \ref{uconv}. For too
large a value ($\epsilon=1.4$ in this case), the requirement in
Theorem~\ref{uconv} that $B\epsilon<1$ ceases to hold, and the method
fails to converge.  This poses a limitation on the applicability of
this technique in that it may break down for very large-amplitude
interfaces. It is important to realize, however, that the choice of
$\epsilon$ itself is immaterial; the true determinant of convergence
is $\norm{\eta}_{H^s(D)}$ (since $B\epsilon=E_2\|\eta\|_{H^s(D)}$).
The radius of convergence can be increased by using the Pade
approximation to the DNO field expansion \cite{fang07}, but even this
approach may fail to converge for sufficiently large $\epsilon$.

In practice, the method converges for surprisingly large-amplitude
waves. Two examples are shown in Figures~\ref{fig:contourPlotsE8}
and~\ref{fig:stand1d}. Figure~\ref{fig:contourPlotsE8} shows the
interface $\eta(\rho,\theta)$ and Dirichlet data $q(\rho,\theta)$
corresponding to $\epsilon=0.8$ in Figure~\ref{erpleps}.  The function
$q(\rho,\theta)$ in (\ref{eq:q:exa}) is the product of a mildly
oscillatory function $J_3(a_{32}\rho)\cos(3\theta)$ and a hyperbolic
cosine function that decays by a factor of $0.000575$ from the right
side of the disk (where $\eta$ is largest) to the left side. Thus, the
oscillations in $q(\rho,\theta)$ in the left half of the unit disk are
strongly suppressed in comparison to those in the right half. The
gradient of $\eta$ in this example has magnitude $0.74$ at the origin,
so this wave profile is far from flat; nevertheless, it is still
well-inside the radius of convergence of the Dirichlet-Neumann
operator.

Figure~\ref{fig:stand1d} shows similar behavior in one-dimension. A
large-amplitude standing water wave of unit mean depth is evolved over
a quarter-period, $0\le t\le T/4$.  At the times shown, we computed
the TFE expansion of the Neumann data for the given Dirichlet data
$q(x,t)$ on the wave surface $\eta(x,t)$ and compared it to a boundary
integral computation of the Neumann data. At all times of the
simulation, the error converges rapidly to zero as the TFE order $K$
increases. The decay rate is fastest when the wave amplitude is small
(near $t=0$), and slowest when it is large (near $t=T/4$).  When
$t=T/4$, the wave comes to rest and the (numerical) Dirichlet data
obtained by evolving the water wave is zero to roundoff accuracy. The
resulting $q$ is taken as an exact (quite oscillatory,
small-amplitude) initial condition in both the TFE method and the
boundary integral method, and we still obtain rapid convergence of the
TFE solution to the boundary integral solution as $K$ increases from 0
to 65, just shifted down by a factor of $10^{-14}$ due to the small
size of $q$.

\begin{figure}[tbph]
	\centering
	\includegraphics[width=\linewidth]{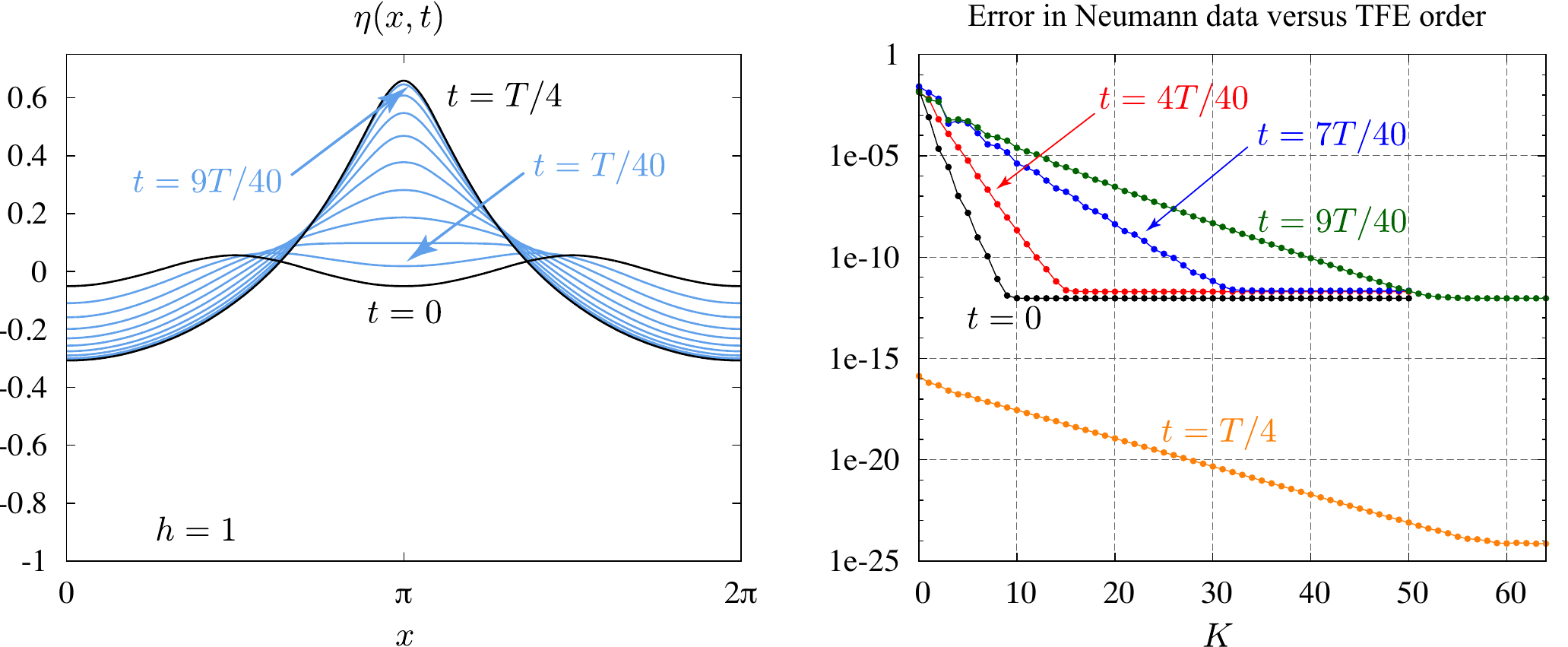}
	\caption{One-dimensional example showing that the transformed field expansion
		converges over the entire evolution of a large-amplitude
		standing water wave. (left) Evolution of solution B of Fig.~2
		in \cite{wilk:yu:2012}, which has period $T=7.240$. Here we
		replaced the boundary integral method in \cite{wilk:yu:2012}
		with the one-dimensional TFE code described in
		\cite{wilk5w}. (right) The error in the
		computed Neumann data at each TFE order is measured
		against the boundary integral solution at the times shown.
	}
	\label{fig:stand1d}
\end{figure}

%However, for systems that are supposed to be stable, this technique
%serves as an extremely efficient simulation tool.
%We present a qualitative validation of our technique.

We conclude with a water wave calculation in a cylindrical geometry to
demonstrate the effectiveness of the TFE method as a simulation tool
in this setting.  Consider a fluid initially at rest with
$\eta(\rho,\theta)|_{t = 0} = 0.05\rho e^{-15 \rho^2}\cos(\theta)$. Using our
DNO solver in conjunction with a time-integration technique for the
system (\ref{etaeqn},\ref{qeqn}), we can numerically evolve the system
and validate it qualitatively (see \cite{qadeerthesis,faradaypaper}
for details). Figure \ref{sim2} shows various stages in the
progression of the fluid. In particular, one can note the collapse of
the crest and trough and their outward dispersion and reflection after
striking the lateral boundaries.

\begin{figure}[tbph]
	\def \scl {0.15}
	\centering
	{
		\includegraphics[scale=\scl]{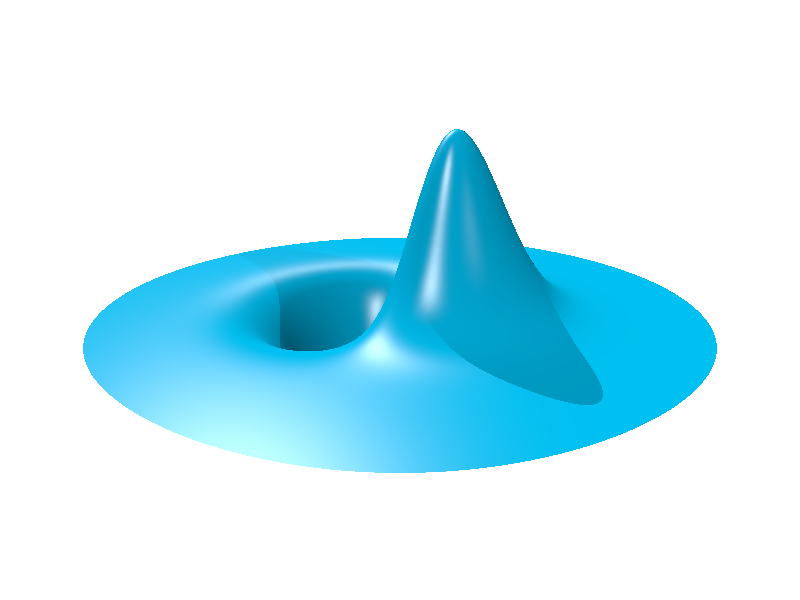}
		\includegraphics[scale=\scl]{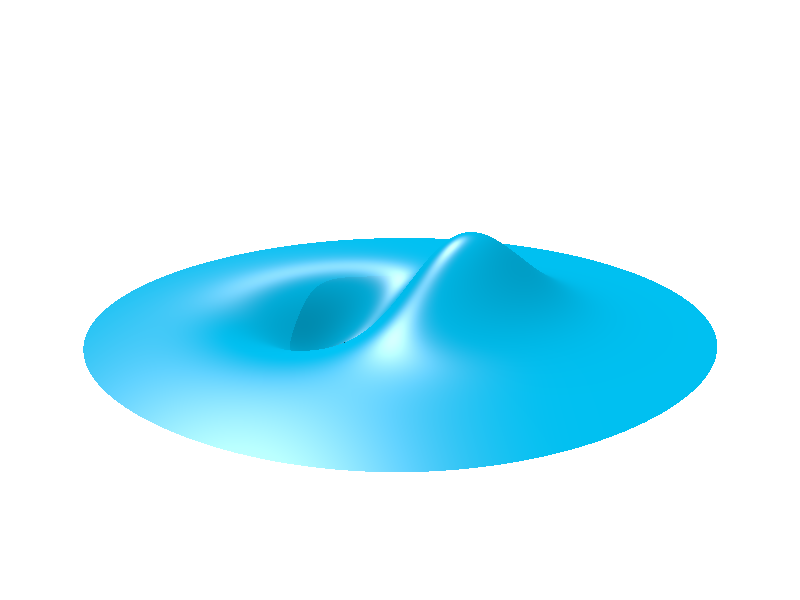}
		\includegraphics[scale=\scl]{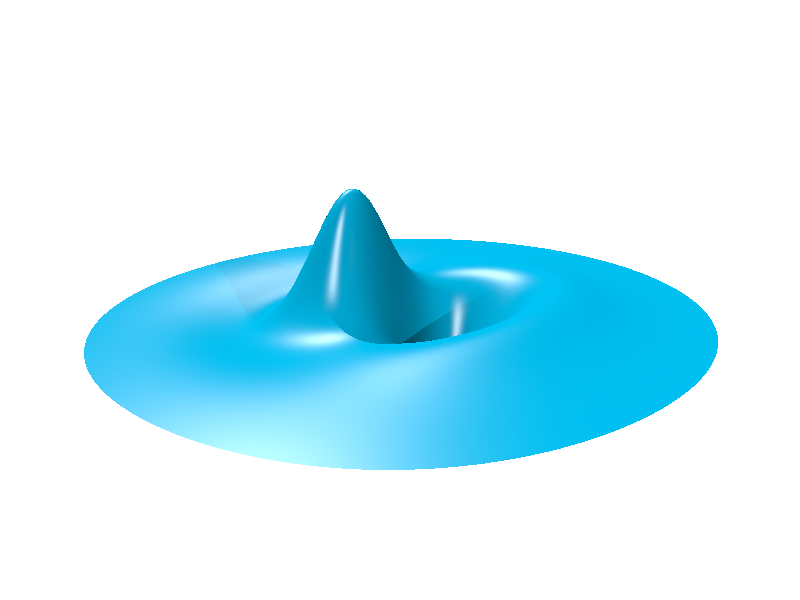}
		\includegraphics[scale=\scl]{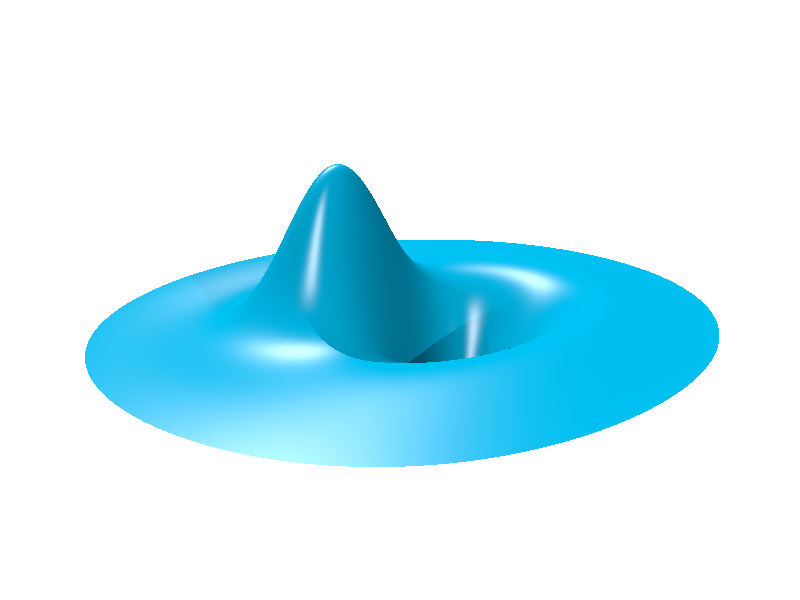}
		\includegraphics[scale=\scl]{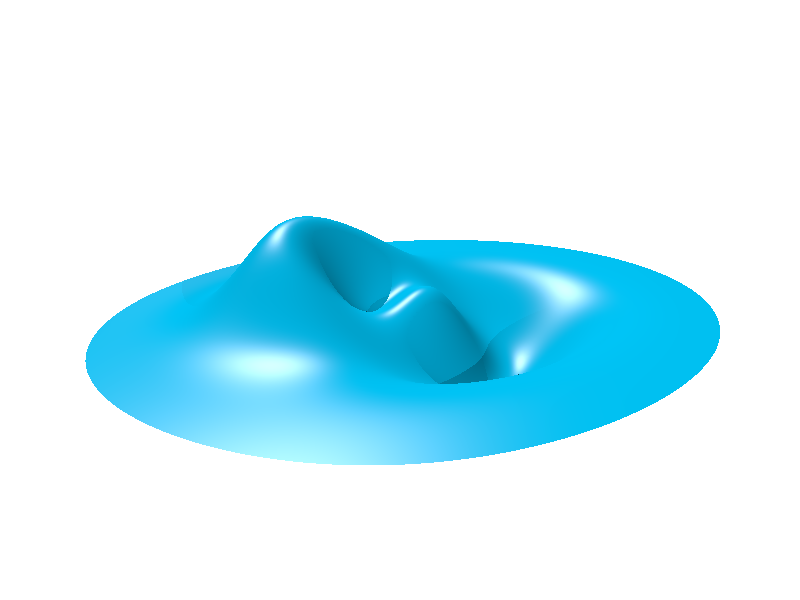}
		\includegraphics[scale=\scl]{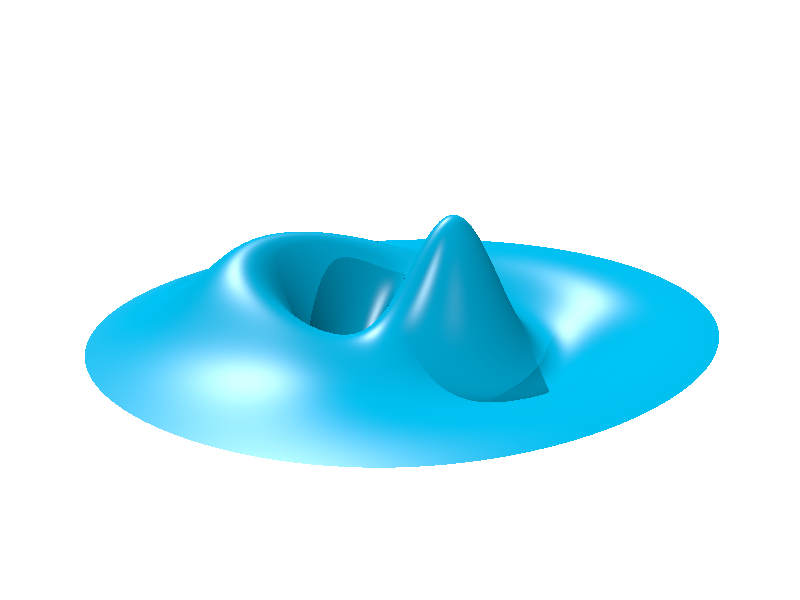}
		\includegraphics[scale=\scl]{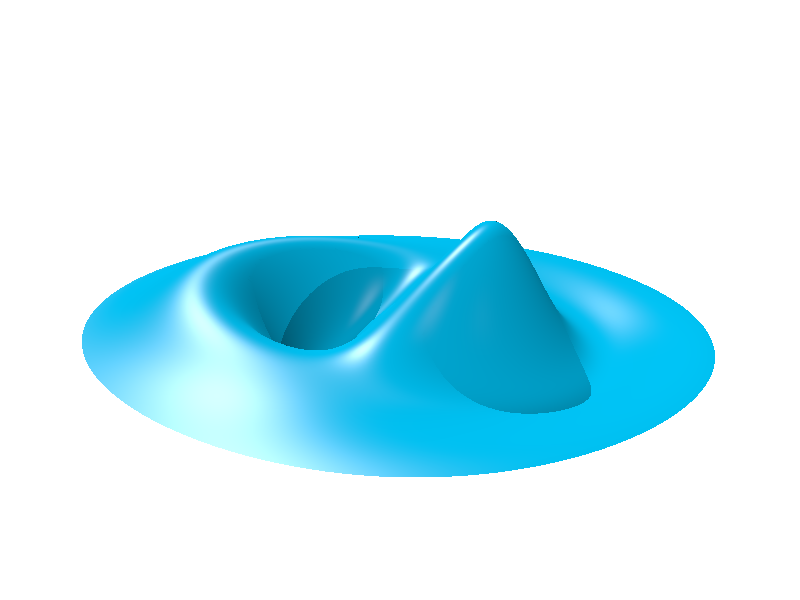}
		\includegraphics[scale=\scl]{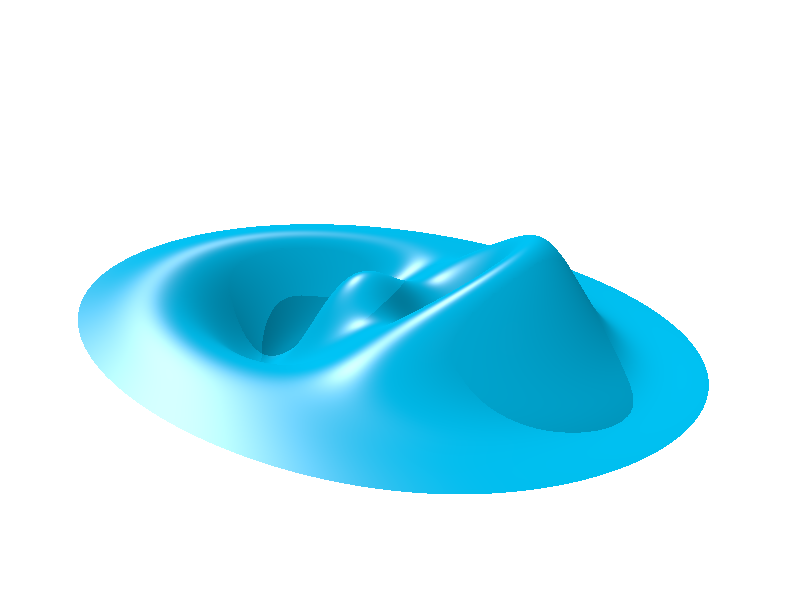}
		\includegraphics[scale=\scl]{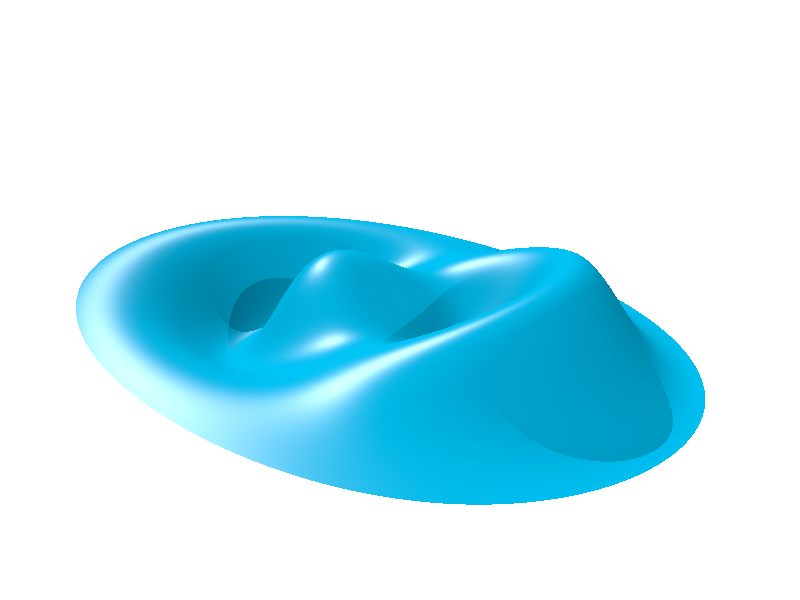}
		\includegraphics[scale=\scl]{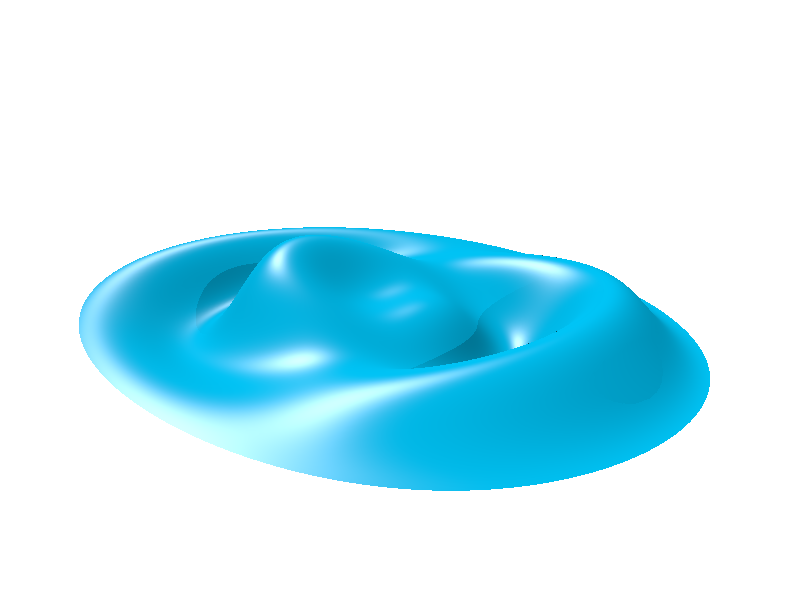}
		\includegraphics[scale=\scl]{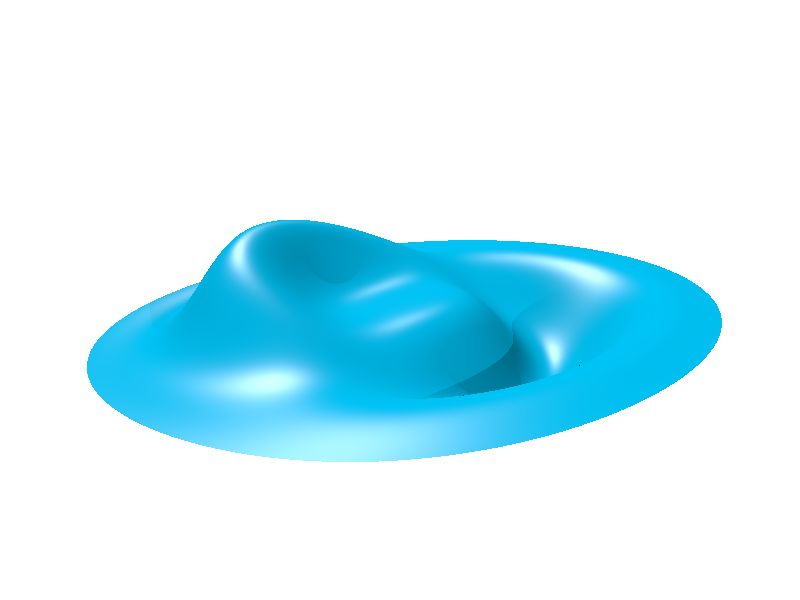}
		\includegraphics[scale=\scl]{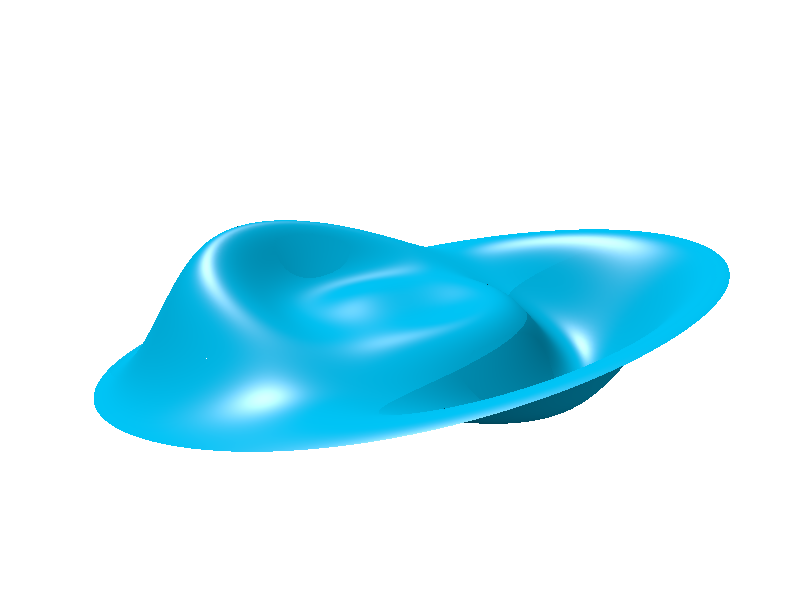}
		\includegraphics[scale=\scl]{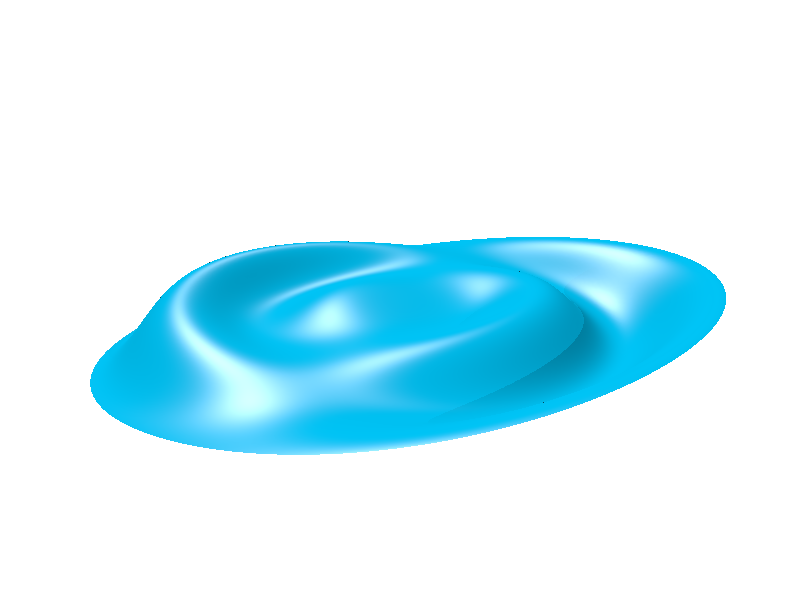}
		\includegraphics[scale=\scl]{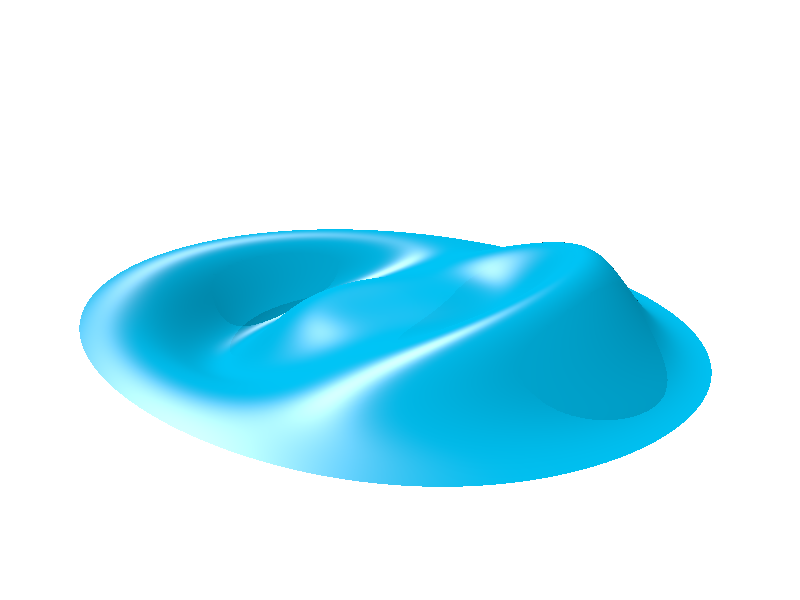}
		\includegraphics[scale=\scl]{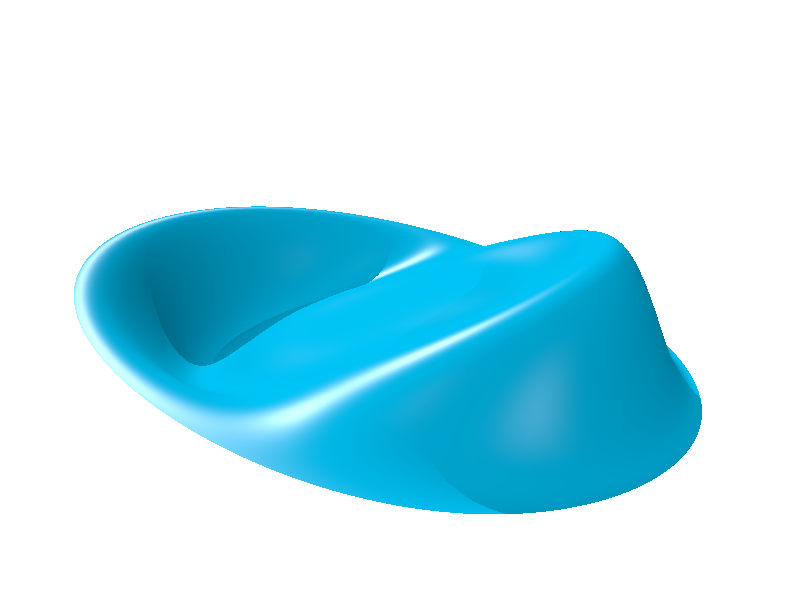}
	}
	\caption{Evolution of the interface at $t = 0,1/80,2/80,\hdots, 14/80$. We use $M = 4, J = 20, N = 40, K = 2, h = 0.5, \epsilon = 0.01$ and RK4 for the time evolution with a time-step of $\Delta t = 1/1200$.} \label{sim2}
\end{figure}

\section{Conclusion}
We have presented a new technique for computing the DNO for Laplace's equation on a cylinder of finite depth. Its novelty lies in the fact it is primarily tailored for a 3D geometry and does not rely on periodic boundary conditions to avoid dealing with the fluid-boundary interactions. Hence, this method represents a major step-up from the methods that are currently in use. In addition, it is easily applicable to other regular domains in 3D, e.g., prisms, parallelepipeds, etc. and may also find use in domains with structured irregularities (for instance \cite{nachmil}). A similar approach may be developed for the computation of the DNO on a sphere, as formulated by \cite{llave}.

The development of the technique is a generalization of the three-term recurrence formulation presented in \cite{nichollshops,nicreit04}. However, using tools from differential geometry, we were able to significantly cut down on the tedious algebra. We allied the formulation with a particular choice of basis functions on the cylinder that are amenable to the various operations that arise ubiquitously and hence obtained a fast algorithm. The preference for Jacobi polynomials over Bessel functions in the radial direction is borne out of the need for faster manipulations and greater accuracy. 
%It is notable that the infinite series representation for the latter is of the form \eqn{J_{m}(r) = r^{|m|}(c_0 + c_1r^2 + c_2r^4 + \hdots)\nonumber} and that the Jacobi polynomials, with the orthogonality weight, also possess a similar form, namely, $r^{|m|}P^{(0,|m|)}_n(2r^2 - 1)$. Thus, these polynomials serve as effective substitutes for Bessel functions but with superior approximation properties for representing functions on the disc. In addition, Jacobi polynomials come with quadrature rules, orthogonality, recurrence relations and closed forms of eigenvalues and derivatives, among other desirable properties. These are harder to obtain for Bessel functions. For instance, orthogonality only occurs when the radial axis is scaled by the zeros of the Bessel functions (or their derivatives), which in turn need to be computed beforehand.     

A similar advantage is gained by the use of Lagrange polynomials with respect to the Chebyshev-Lobatto nodes along the $z$-axis in place of hyperbolic functions. These polynomials possess a fast transformation to Chebyshev polynomials, which can in turn be differentiated and evaluated accurately at arbitrary points; this comes in handy when, for instance, setting up the quadrature matrices. While not affecting the computational cost and accuracy, they allow us to apply the boundary conditions more easily than would be possible for other function families. The structure of the basis functions therefore yields a fast, well-conditioned solver and ensures that implementation boils down to a sequence of linear algebra operations that can be performed rapidly using BLAS and LAPACK routines.

The analysis of Zernike polynomials presented here illustrates their approximation properties on a disc. The function class $H^s(D)$ is broad enough to describe the majority of the phenomena encountered in water wave problems. The approximation estimate also leads to a rigorous convergence proof for the TFE method. In particular, it establishes that the convergence hinges only on $\norm{\eta}_{H^s(D)}$: the strength of the error decay, as well as possible divergence, is dependent entirely on the interface shape $\eta$. The result also shows that the parameter $\epsilon$ is merely a book-keeping device to help group together terms of the same order when deriving the recurrence formulas. We can conclude that this technique yields a fast and accurate solver for nonlinear water-wave equations when the amplitude does not grow too large.

Since this approach relies on the potential form of the water-wave equations, it disallows dissipation as it appears in the Navier-Stokes equations. To counter this, one can use models of potential viscous flows that artificially introduce dissipation. These have been noted to lead to correct results in the linear wave limit \cite{kaknic09} and have found use in various applications \cite{milrios}. The TFE technique lends itself to these models in a fairly straightforward manner \cite{qadeerthesis,faradaypaper}.

\begin{appendices}
\section{Fractional Sobolev Spaces and Interpolation}\label{appa}
In this appendix, we define the norm on $H^s(D)$ and state the
interpolation result that we used in Theorem \ref{discerr}. When $m =
2k$ is an even integer, we define \eqn{ \norm{u}^2_{H^m(D)} &=&
  \sum_{|\alpha| \leq m} \norm{\partial^\alpha u
  }^2_{L^2(D)}, \label{hmnorm} } and for $s = 2(k+\nu)$ for $0 < \nu <
1$, we define \eqn{ \norm{u}_{H^s(D)} &=&
  \norm{S_k^{1-\nu}u}_{H^{2k+2}(D)}, \label{hsnorm} } where $S_k$ is
the unique positive square root of the compact self-adjoint operator
$S_k^2$ defined by
$$\ip{S_k^2u,v}_{H^{2k+2}(D)} = \ip{u,v}_{H^{2k}(D)}$$
for $u,v \in H^{2k+2}(D)$. As $\nu \to 0$, $\norm{u}_{H^s(D)} \to \norm{u}_{H^{2k}(D)}$; as $\nu \to 1$, $\norm{u}_{H^s(D)} \to \norm{u}_{H^{2k+2}(D)}$; and for $\nu = \frac{1}{2}$, $\norm{u}_{H^s(D)}$ is equivalent to the norm in (\ref{hmnorm}) with $m = 2k+1$ (see \cite{bmaday97,chandler_int,mclean}). 

The principal theorem of interpolation \cite{bmaday97} states in our case that if $T:H^{2k}(D) \to L^2(D)$ is bounded with norm $A$ and $T|_{H^{2k+2}(D)}$ is bounded with norm $B$, then $T|_{H^s(D)}$ is bounded with norm $\leq A^{1-\nu}B^{\nu}$, where $s$ and $\nu$ are as above.

\end{appendices}

\end{document}